\documentclass[preprint,10pt]{elsarticle}

\usepackage[T1]{fontenc}
\usepackage{lmodern}
\usepackage{amsmath,amssymb,amsfonts,amsthm,mathtools}
\usepackage{graphicx}
\usepackage{stmaryrd}
\usepackage{xcolor}
\usepackage{aliascnt}
\usepackage{bm}
\usepackage[
  colorlinks=true,
  linkcolor=blue,
  citecolor=blue,
  urlcolor=blue
]{hyperref}
\usepackage[capitalise,noabbrev,nameinlink]{cleveref}
\usepackage{xfrac}
\usepackage{tikz,tools}
\usepackage{pgfplots}
\pgfplotsset{compat=1.18}
\usetikzlibrary{spy,patterns}
\usepackage{subfigure}
\usepackage{pifont}
\usepackage{booktabs}

\definecolor{MidnightBlue}{HTML}{191970}
\definecolor{ForestGreen}{rgb}{0.0, 0.27, 0.13}
\definecolor{carmine}{rgb}{0.59, 0.0, 0.09}
\definecolor{burntorange}{rgb}{0.8, 0.33, 0.0}
\definecolor{BrickRed}{rgb}{0.8, 0.25, 0.33}

\newcommand{\safeincludegraphics}[2][]{%
  \IfFileExists{#2}{\includegraphics[#1]{#2}}{\fbox{\texttt{Missing file: #2}}}%
}

\theoremstyle{plain}
\newtheorem{theorem}{Theorem}[section]

\newaliascnt{lemma}{theorem}
\newtheorem{lemma}[lemma]{Lemma}
\aliascntresetthe{lemma}

\newaliascnt{corollary}{theorem}
\newtheorem{corollary}[corollary]{Corollary}
\aliascntresetthe{corollary}

\newaliascnt{claim}{theorem}

\aliascntresetthe{claim}

\theoremstyle{definition}
\newaliascnt{definition}{theorem}
\newtheorem{definition}[definition]{Definition}
\aliascntresetthe{definition}

\newaliascnt{hypothesis}{theorem}

\aliascntresetthe{hypothesis}

\theoremstyle{remark}
\newaliascnt{remark}{theorem}
\newtheorem{remark}[remark]{Remark}
\aliascntresetthe{remark}

\crefname{theorem}{Theorem}{Theorems}
\Crefname{theorem}{Theorem}{Theorems}
\crefname{lemma}{Lemma}{Lemmas}
\Crefname{lemma}{Lemma}{Lemmas}
\crefname{claim}{Claim}{Claims}
\Crefname{claim}{Claim}{Claims}
\crefname{definition}{Definition}{Definitions}
\Crefname{definition}{Definition}{Definitions}
\crefname{hypothesis}{Hypothesis}{Hypotheses}
\Crefname{hypothesis}{Hypothesis}{Hypotheses}
\crefname{remark}{Remark}{Remarks}
\Crefname{remark}{Remark}{Remarks}

\begin{document}

\begin{frontmatter}

\title{A Unified Transmissibility-Based Interior Penalty DG Method for Heterogeneous and Anisotropic Diffusion}

\author[ovpf]{Gregory Etangsale}
\ead{etangsale@ipgp.fr}

\author[piment]{Vincent Fontaine\corref{cor1}}
\cortext[cor1]{Corresponding and Principal author}
\ead{vincent.fontaine@univ-reunion.fr}

\author[ites]{Anis Younes}
\ead{younes@unistra.fr}

\affiliation[ovpf]{organization={Observatoire volcanologique du Piton de la Fournaise, F-97418}, addressline={La R\'eunion}, country={France}}
\affiliation[piment]{organization={PIMENT, University of La R\'eunion, ESIROI}, addressline={La R\'eunion}, country={France}}
\affiliation[ites]{organization={ITES, University of Strasbourg, CNRS}, addressline={Strasbourg}, country={France}}

\begin{abstract}
We derive a primal discontinuous Galerkin (DG) formulation for heterogeneous and anisotropic diffusion, obtained by exact algebraic elimination of the skeletal unknown in a compact hybridized interior penalty (H-IP) method. The resulting Unified Interior Penalty DG (UIP-DG) scheme involves transmissibility-based weights inherited from the hybrid formulation, together with two stabilization terms acting respectively on the primal jump and on the jump of the normal diffusive flux. These penalties scale, respectively, with the harmonic mean and with the inverse arithmetic mean of the face-wise transmissibilities. This construction provides a unified perspective on several interior penalty approaches previously introduced independently, while yielding a robust method with stability properties independent of the diffusion contrast and anisotropy. We prove consistency, coercivity, and boundedness of the formulation, and derive quasi-optimal energy-norm \emph{a priori} error estimates for all variants. Numerical experiments confirm the theoretical claims.
\end{abstract}


\end{frontmatter}

\section{Introduction}

Robust discretization of diffusion problems with strongly heterogeneous and anisotropic coefficients remains a central challenge in computational science. In such settings, permeability contrasts may span several orders of magnitude across element interfaces, so that the accurate scaling of interface terms --- numerical fluxes, weighted averages, and penalty parameters --- is critical for stability and robustness. Interior penalty discontinuous Galerkin (IPDG) methods are natural candidates due to their flexibility and local conservation properties \cite{Riviere08bookDG,di2011mathematical,ABCMUnified}. However, standard IPDG formulations based on arithmetic averages and single-valued penalties may suffer from loss of robustness and poor conditioning in high-contrast regimes~\cite{DryjaCMAME2003,ESZIMA08}. Flux-based penalty techniques originate from the classical interior penalty constructions of~\cite{Douglas1976InteriorPP}.

Hybridized DG (HDG) methods provide an alternative framework that reduces the number of globally coupled degrees of freedom via static condensation, while delivering accurate traces on the mesh skeleton suitable for post-processing and superconvergence~\cite{CGLSIAM09,LehrenfeldTH,Cockburn2016,Nguyen2009}. Within this family, compact hybridized interior penalty (H-IP) schemes constitute a particularly effective class for diffusion: element unknowns are eliminated locally, leading to a global system posed solely on the mesh skeleton, and local flux conservation is enforced by construction. Their stabilization naturally involves side-dependent (double-valued) parameters that scale with face-wise transmissibilities, which is precisely the scaling required for robustness in high-contrast and anisotropic regimes.

In heterogeneous and anisotropic diffusion, the fundamental distinction between conventional IPDG schemes and compact H-IP methods lies in the structure of their interface stabilization. Standard IPDG formulations rely on single-valued, \emph{ad hoc} face parameters and associated consistency weights, typically constructed from arithmetic or harmonic averages of the neighboring diffusion tensors. In contrast, H-IP methods naturally produce side-dependent (double-valued) stabilization operators through the local flux-conservation condition, with both weights and penalty scalings uniquely determined by face-wise transmissibilities, \ie by the natural scaling of the normal diffusive flux across the interface.

Existing works on DG reductions of H-IP methods~\cite{wells2011analysis,FabienNM2019} impose a restrictive single-valued assumption on the stabilization parameter. Under this constraint, the transmissibility-based weighting inherent to the hybrid formulation is lost, and the resulting DG formulation reduces to an IP-type scheme with arithmetic-mean consistency weights, a standard primal jump penalty, and an additional flux-jump stabilization term. Although the flux-jump term provides additional interface control, arithmetic averaging fails to capture the correct diffusion scaling across interfaces, leading to a coercivity constant that deteriorates as the diffusion contrast increases \cite{DryjaCMAME2003,BurmanSIAM2006,EtangsaleCAMWA2022}. This reflects the mismatch between arithmetic averaging and the transmissibility-based scaling of the continuous flux. Allowing genuinely double-valued, transmissibility-based weights is therefore essential for contrast-robustness. To the best of our knowledge, an exact primal DG reduction retaining this double-valued transmissibility structure, with explicit closed-form parameters, is not available in the literature.

\paragraph{Contributions}
This work makes three main contributions.

\begin{itemize}

\item \textbf{A unified primal DG formulation.}
We start from a compact H-IP scheme with genuinely double-valued stabilization and eliminate the skeletal unknown in closed form. The resulting UIP-DG method is therefore not designed \emph{ad hoc}: its interface structure --- including transmissibility-based weights and both primal- and flux-jump penalizations --- follows directly from the underlying hybridized formulation.

\item \textbf{Recovery of intrinsic numerical traces.}
A distinctive feature of UIP-DG is that, although formulated purely in terms of the primal unknown, it still admits the \emph{a posteriori} reconstruction of a unique pair of numerical traces $(\dtu,\dtsg)$ associated with the eliminated hybrid variable. In particular, the numerical flux is continuous across interior interfaces for all $\epsilon\in\{0,\pm1\}$, which ensures local conservation and avoids any \textit{ad hoc} flux reconstruction.

\item \textbf{Contrast-robust analysis and optimal approximation properties.}
We establish coercivity and boundedness in an energy norm tailored to heterogeneous and anisotropic diffusion, with constants robust with respect to coefficient jumps. This yields quasi-optimal error estimates in the energy norm for all variants, and, in the symmetric case, an optimal $\LL{2}$-error estimate by duality.

\end{itemize}

The material is organized as follows. \cref{S:2} introduces the mesh notation, broken polynomial spaces, and the discrete trace operators --- in particular, the key identity relation (\cref{lemma:identity_rel}) that underlies the entire DG reduction. \cref{S:3} presents the compact H-IP formulation, carries out the
elimination of the skeletal trace unknown, and states the resulting UIP-DG method together with its connections to existing DG schemes. \cref{S:4} provides the full stability and \emph{a priori} error analysis. Numerical results are reported and discussed in \cref{S:5}.

\section{Preliminaries}
\label{S:2}

\subsection{Model problem}

Let $\domain\subset\R^d$ ($d\in\{2,3\}$) be a bounded polyhedral domain with Lipschitz boundary $\bnddom$. We consider the heterogeneous, anisotropic diffusion problem with homogeneous Dirichlet boundary conditions:
\begin{equation}
  \label{model_problem_strong}
  -\dvg{(\bkap\grd{u})}=f \quad \text{in }\domain,
  \qquad u=0 \quad \text{on }\bnddom,
\end{equation}
where $f\in\lebesgue{2}{\domain}$ is a source term and $\bkap:\domain\to\R^{d\times d}$ is a symmetric diffusion tensor, possibly discontinuous across material interfaces. We assume $\bkap\in[\lebesgue{\infty}{\domain}]^{d\times d}$ and that there exist constants $0<\kappa_{\min}\le\kappa_{\max}<\infty$ such that
\begin{equation}
  \label{ellipticity}
  \kappa_{\min}|\xi|^2
  \;\le\;
  \xi^\top \bkap(x)\,\xi
  \;\le\;
  \kappa_{\max}|\xi|^2,
  \qquad
  \forall\,\xi\in\R^d,
  \quad
  \text{for a.e. }x\in\domain.
\end{equation}
The ratio $\kappa_{\max}/\kappa_{\min}$, referred to as the \emph{contrast ratio}, may be arbitrarily large; controlling this regime is a primary motivation for the UIP-DG method developed in \cref{S:3}. By the Lax--Milgram lemma, problem~\eqref{model_problem_strong} admits a unique weak solution $u\in V\eqbydef\hilbert{1}{\domain}{0}$ satisfying
\begin{equation}
  \label{weak_model_problem}
  \inerprod{\bkap\grd{u}}{\grd{v}}{0}{\domain}
  =\inerprod{f}{v}{0}{\domain},
  \qquad \forall\,v\in V,
\end{equation}
where $\inerprod{\cdot}{\cdot}{0}{\domain}$ denotes the $\lebesgue{2}{\domain}$ inner product. 


\subsection{Discrete setting}

\paragraph{Mesh and skeleton}
Let $\{\Th\}_{h>0}$ be a shape-regular family of conforming affine simplicial meshes of $\domain$, with global meshsize $h\eqbydef\max_{\e{}\in\Th}h_{\e{}}$ and local diameter $h_{\e{}}\eqbydef\diam{\e{}}$. We assume $h\le 1$ without loss of generality. For $d=2$ (resp.\ $d=3$), an \emph{interface} is an edge (resp.\ a face). The mesh skeleton is partitioned as $\Fh\eqbydef\Fhi\cup\Fhb$, where $\Fhi$ and $\Fhb$ collect the interior and boundary interfaces, respectively. For each element $\e{}\in\Th$, we denote by $\FE$ the set of its interfaces, by $\vect{n}_{\e{},\f{}}$ the unit outward normal to $\f{}\in\FE$, and set
\[
  \bnd{\Th}
  \eqbydef
  \bigcup_{\e{}\in\Th}\FE,
  \qquad
  \eta_{\e{}}\eqbydef\card{\FE},
  \qquad
  \eta_0\eqbydef\max_{\e{}\in\Th}\eta_{\e{}}.
\]
Note that $\bnd{\Th}$ is a multiset (each interior interface appears twice, once from each side), whereas $\Fh$ collects each interface once; this distinction is reflected in the inner products introduced below. For a measurable set $X$, $|X|$ denotes its $d$- or $(d-1)$-dimensional Lebesgue measure. 

\paragraph{Functional spaces}
For a polyhedral domain $D\subset\R^d$, we use the standard Sobolev spaces $\hilbert{s}{D}{}$ with norm $\norm{\cdot}{s}{D}$ and seminorm $\seminorm{\cdot}{s}{D}$. For $s=0$, $\hilbert{0}{D}{}=\lebesgue{2}{D}$ with inner product $\inerprod{\cdot}{\cdot}{0}{D}$; on boundaries we write $\dualprod{\cdot}{\cdot}{0}{\bnd{D}}$ for the $\lebesgue{2}{\bnd{D}}$ inner product. For a mesh $\Th$, the \emph{broken Sobolev space} is
\[
  \hilbert{s}{\Th}{}
  \eqbydef
  \bigl\{v\in\lebesgue{2}{\domain}:
  v|_{\e{}}\in\hilbert{s}{\e{}}{}\ \forall\e{}\in\Th\bigr\},
\]
equipped with the broken norm $\norm{\cdot}{s}{\Th}$ and the broken gradient $\grdh{}$. To handle interface integrals uniformly, we introduce the aggregated $\LL{2}$-inner products:
\begin{equation*}
  \dualprod{\cdot}{\cdot}{0}{\bnd{\Th}}
  \eqbydef
  \sumT\sumFe\dualprod{\cdot}{\cdot}{0}{\f{}},
  \qquad
  \dualprod{\cdot}{\cdot}{0}{\Fh}
  \eqbydef
  \sumF\dualprod{\cdot}{\cdot}{0}{\f{}},
\end{equation*}
with corresponding norms $\norm{\cdot}{0}{\bnd{\Th}}$ and $\norm{\cdot}{0}{\Fh}$. 

\paragraph{Polynomial spaces and composite HDG space}
Let $k\ge 1$. The broken polynomial spaces are
\begin{subequations}
\begin{align}
  \Pk{k}{\Th}
  &\eqbydef
  \bigl\{\dv\in\lebesgue{2}{\domain}:
  \dv|_{\e{}}\in\Pk{k}{\e{}},\ \forall\e{}\in\Th\bigr\},\\
  \Pk{k}{\Fh}
  &\eqbydef
  \bigl\{\dtv\in\lebesgue{2}{\Fh}:
  \dtv|_{\f{}}\in\Pk{k}{\f{}},\ \forall\f{}\in\Fh\bigr\},
\end{align}
\end{subequations}
where $\Pk{k}{X}$ denotes the space of polynomials of \emph{total degree at most $k$} on $X$. We set
\[
  \Vh\eqbydef\Pk{k}{\Th},
  \qquad
  \TVh{}\eqbydef
  \bigl\{\tv\in\Pk{k}{\Fh}:
  \tv|_{\f{}}=0,\ \forall\f{}\in\Fhb\bigr\},
\]
where the homogeneous condition on $\Fhb$ encodes the Dirichlet boundary condition at the discrete skeletal level. The \emph{composite HDG space} is then
\[
  \CVh{}\eqbydef\Vh\times\TVh{},
\]
whose elements $\cdv\eqbydef(\dv,\dtv)\in\CVh{}$ pair a broken polynomial $\dv$ with a single-valued trace $\dtv$ on the skeleton. 
We further assume that the diffusion tensor $\bkap$ is piecewise constant with respect to the mesh, \ie $\bkap \in \left[\Pk{0}{\Th}\right]^{d\times d}$. 

\subsection{Discrete trace operators}
\label{S:2.3}

To bridge the H-IP and primal DG formulations, we introduce three trace operators: the standard DG jump and weighted mean, acting on broken functions over $\Fh$, and the local HDG-jump, acting elementwise over $\bnd{\Th}$. A central role is played by a \emph{skeletal-weighting function} $\vomega:\Fh\to\R^2$, which assigns a pair of non-negative weights to each interface. Specifically, for each interior interface $\f{}\in\Fhi$ shared by elements $\e{1}$ and $\e{2}$ (\ie
$\f{}=\bnd{\e{1}}\cap\bnd{\e{2}}$), we set
\begin{equation}
  \label{weight_function}
  \vomega|_{\f{}}
  \eqbydef(\omega_1,\omega_2),
  \qquad
  \omega_i\eqbydef\omega|_{\e{i},\f{}},
  \quad i=1,2,
\end{equation}
with $\omega_1+\omega_2=1$ a.e.\ on $\f{}$. On boundary faces $\f{}\in\Fhb$, we set $(\omega_1,\omega_2)\eqbydef(1,0)$. 

\begin{definition}[Discrete trace operators]
\label{def_traces_operators}
Let $s\ge 1$ and $\boldsymbol{\varphi}\eqbydef(\varphi,\hat\varphi)
\in\hilbert{s}{\Th}{}\times\lebesgue{2}{\Fh}$. For $\f{}\in\Fhi$ shared by $\e{1}$ and $\e{2}$, let $\varphi_i\eqbydef\varphi|_{\e{i},\f{}}$ and $\vect{n}_i\eqbydef\vect{n}_{\e{i},\f{}}$ denote the traces of $\varphi$ and the outward unit normal to $\f{}$ from $\e{i}$, respectively.
\begin{itemize}
  \item \textbf{DG jump.} The normal jump operator
  $\jump{\cdot}:\hilbert{s}{\Th}{}\to[\lebesgue{2}{\Fh}]^d$ is defined by
  \begin{equation}
    \label{sub:wjump}
    \jump{\varphi}|_{\f{}}
    \eqbydef
    \begin{cases}
      \varphi_1\vect{n}_1+\varphi_2\vect{n}_2,
      & \f{}\in\Fhi,\\[4pt]
      \varphi|_{\e{},\f{}}\,\vect{n}_{\e{},\f{}},
      & \f{}\in\Fhb.
    \end{cases}
  \end{equation}
  \item \textbf{Weighted DG mean.} Given a skeletal-weighting function $\vomega$, the weighted-mean operator
  $\wmean{\cdot}{\vomega}:\hilbert{s}{\Th}{}\to\lebesgue{2}{\Fh}$ is
  \begin{equation}
    \label{sub:wmean}
    \wmean{\varphi}{\vomega}|_{\f{}}
    \eqbydef
    \begin{cases}
      \omega_1\varphi_1+\omega_2\varphi_2,
      & \f{}\in\Fhi,\\[4pt]
      \varphi|_{\e{},\f{}},
      & \f{}\in\Fhb.
    \end{cases}
  \end{equation}
  The \emph{conjugate} weighted mean $\cwmean{\cdot}{\vomega}$ is obtained by swapping the weights: $\cwmean{\varphi}{\vomega}|_{\f{}}\eqbydef\omega_2\varphi_1+\omega_1\varphi_2$ on $\f{}\in\Fhi$, and coincides with $\wmean{\cdot}{\vomega}$ on $\Fhb$. Setting $\omega_1=\omega_2=\shalf$ recovers the arithmetic mean $\mean{\varphi}\eqbydef\shalf(\varphi_1+\varphi_2)$ in both cases.
  \item \textbf{HDG local jump.} The HDG-jump operator   $\tjump{\cdot}:\hilbert{s}{\Th}{}\times\lebesgue{2}{\Fh}  \to[\lebesgue{2}{\bnd{\Th}}]^d$ measures the local discrepancy between the broken function and its skeletal counterpart, and is defined elementwise by
  \begin{equation}
    \label{sub:hdgjump}
    \tjump{\boldsymbol{\varphi}}|_{\e{},\f{}}
    \eqbydef
    \bigl(\varphi|_{\e{},\f{}}-\hat\varphi|_{\f{}}\bigr)\vect{n}_{\e{},\f{}},
    \qquad
    \forall\e{}\in\Th,\ \forall\f{}\in\FE.
  \end{equation}
  Unlike $\jump{\cdot}$, which is single-valued on $\Fh$, the HDG jump  $\tjump{\cdot}$ is double-valued on interior interfaces.
\end{itemize}
The DG operators $\jump{\cdot}$, $\wmean{\cdot}{\vomega}$ and $\cwmean{\cdot}{\vomega}$ extend componentwise to vector fields.
\end{definition}
The following identity is the key algebraic ingredient underlying the proposed DG reduction. It provides an exact rewriting of boundary integrals over $\bnd{\Th}$ as interface integrals over $\Fh$ in terms of DG trace operators, for arbitrary weights $\vomega$.
\begin{lemma}[Identity relation]
\label{lemma:identity_rel}
Let $\vect{b}\in[\hilbert{1}{\Th}{}]^d$ and $\boldsymbol{\varphi}=(\varphi,
\hat\varphi)\in\hilbert{1}{\Th}{}\times\lebesgue{2}{\Fh}$ with $\hat\varphi=0$ on
$\Fhb$. Then, for any weighting function $\vomega$,
\begin{equation}
  \label{magic3}
  \dualprod{\vect{b}}{\tjump{\boldsymbol{\varphi}}}{0}{\bnd{\Th}}
  =
  \dualprod{\cwmean{\vect{b}}{\vomega}}{\jump{\varphi}}{0}{\Fh}
  +\dualprod{\jump{\vect{b}}}{\wmean{\varphi}{\vomega}-\hat\varphi}{0}{\Fhi}.
\end{equation}
\end{lemma}
\begin{proof}
Set $\bar\varphi\eqbydef\varphi-\hat\varphi$, so that $\tjump{\boldsymbol{\varphi}}|_{\e{},\f{}}=\bar\varphi|_{\e{},\f{}}\,\vect{n}_{\e{},\f{}}$. By definition of $\dualprod{\cdot}{\cdot}{0}{\bnd{\Th}}$,
\[
  \dualprod{\vect{b}}{\tjump{\boldsymbol{\varphi}}}{0}{\bnd{\Th}}
  =\sumT\sumFe\int_{\f{}}\vect{b}_{\e{},\f{}}\cdot\vect{n}_{\e{},\f{}}\,
  \bar\varphi_{\e{},\f{}}\,\mathrm{d}s
  =\sumF\int_{\f{}}\jump{\vect{b}\bar\varphi}|_{\f{}}\,\mathrm{d}s.
\]
We now treat $\Fhi$ and $\Fhb$ separately. On $\Fhi$, we apply the \emph{product rule}
\begin{equation}
  \label{product_rule}
  \jump{\vect{b}\,\bar\varphi}|_{\f{}}
  =\cwmean{\vect{b}}{\vomega}\jump{\bar\varphi}
  +\jump{\vect{b}}\wmean{\bar\varphi}{\vomega},
\end{equation}
which holds for any skeletal-weighting function $\vomega$. Since $\hat\varphi$ is single-valued on $\Fhi$, we have $\jump{\bar\varphi}=\jump{\varphi}$ and $\wmean{\bar\varphi}{\vomega}=\wmean{\varphi}{\vomega}-\hat\varphi$, so
\[
  \sumFi\int_{\f{}}\jump{\vect{b}\bar\varphi}|_{\f{}}\,\mathrm{d}s
  =\dualprod{\cwmean{\vect{b}}{\vomega}}{\jump{\varphi}}{0}{\Fhi}
  +\dualprod{\jump{\vect{b}}}{\wmean{\varphi}{\vomega}-\hat\varphi}{0}{\Fhi}.
\]
On $\Fhb$, there is a single adjacent element and $\hat\varphi=0$, so
$\bar\varphi=\varphi$ and the jump reduces to a one-sided trace:
\[
  \sumFb\int_{\f{}}\jump{\vect{b}\bar\varphi}|_{\f{}}\,\mathrm{d}s
  =\dualprod{\cwmean{\vect{b}}{\vomega}}{\jump{\varphi}}{0}{\Fhb}.
\]
Combining these above contributions yields~\eqref{magic3}.
\end{proof}

\section{Interior penalty HDG method and its DG-reduction}
\label{S:3}

We first rewrite the compact H-IP formulation in terms of numerical fluxes. Next, we eliminate the trace unknown $\dtu$ from all face contributions, thereby obtaining a purely primal DG formulation posed on $\Vh$.

\subsection{Compact H-IP formulation}

\begin{lemma}[Flux reformulation]
\label{lemma-flux-reform}
The compact H-IP discretization of \eqref{weak_model_problem} reads: find
$\cdu\eqbydef(\du,\dtu)\in\CVh{}$ such that
\begin{equation}
  \label{discrete_problem}
  \dbilin{\boldsymbol{a}^{(\epsilon)}_h}{\cdu}{\cdv}
  = \inerprod{f}{\dv}{0}{\Th},
  \quad\forall\cdv\in\CVh{},
\end{equation}
where $\epsilon\in\{0,\pm1\}$ is the symmetrization parameter. The discrete bilinear form $\boldsymbol{a}^{(\epsilon)}_h:\CVh{}\times\CVh{}\to\R$ admits the flux-based decomposition
\begin{equation*}
  \label{bilinear_form_hdg}
  \dbilin{\boldsymbol{a}^{(\epsilon)}_h}{\cdu}{\cdv}
  \eqbydef
  \inerprod{-\dsg(\du)}{\grdh{\dv}}{0}{\Th}
  +\dualprod{\dtsg(\cdu)}{\tjump{\cdv}}{0}{\bnd{\Th}}
  +\epsilon\dualprod{\dsg(\dv)}{\tjump{\cdu}}{0}{\bnd{\Th}},
\end{equation*}
where $\dsg(\du)\eqbydef-\bkap\grdh{\du}$, and $\dtsg(\cdu)$ is its trace approximation on the mesh skeleton defined by
\begin{equation}
  \label{tracesg}
  \dtsg(\cdu)\eqbydef\dsg(\du)+\tau\tjump{\cdu}
  \quad\text{on}\quad\bnd{\Th},
\end{equation}
with $\tau\in\lebesgue{\infty}{\bnd{\Th}}$ an appropriate stabilization parameter (see, \eg \cref{definition:penalty_parameter}).
\end{lemma}
\begin{proof}
Plugging \eqref{tracesg} into $\boldsymbol{a}^{(\epsilon)}_h$, we recover the usual formulation of the discrete H-IP problem; see \eg \cite[Section~3, p.~91]{EtangsaleCAMWA2022}.
\end{proof}
The parameter $\epsilon$ only affects the adjoint (symmetry) term $\dualprod{\dsg(\dv)}{\tjump{\cdu}}{0}{\bnd{\Th}}$. It controls the global symmetry of the discrete bilinear form and recovers, as special cases, the symmetric scheme (H-SIP, $\epsilon=1$), the incomplete scheme (H-IIP, $\epsilon=0$), and the nonsymmetric scheme (H-NIP, $\epsilon=-1$).
\begin{definition}[H-IP stabilization parameter]
\label{definition:penalty_parameter}
For each element $\e{}\in\Th$ and each face $\f{}\in\FE$, we set
\begin{equation}
  \label{penalty_term}
  \tau|_{\e{},\f{}}
  \eqbydef
  \alpha_0\,\C{T}^2\,\kappa_{\e{},\f{}}\,h_{\e{}}^{-1},
\end{equation}
where $\alpha_0>0$ is a user-chosen constant, $\C{T}$ is the constant in the discrete trace inequality~\eqref{discret_trace_ineq}, and $\kappa_{\e{},\f{}}\eqbydef\vect{n}_{\e{},\f{}}^\top\bkap_{\e{}}\vect{n}_{\e{},\f{}}$ denotes the normal diffusivity on $\f{}$ seen from $\e{}$.
\end{definition}
\begin{remark}[Well-posedness]
Provided that $\alpha_0$ in~\eqref{penalty_term} is chosen sufficiently large, the bilinear form $\boldsymbol{a}^{(\epsilon)}_h$ is coercive on $\CVh{}$. Therefore, by the Lax--Milgram lemma, problem~\eqref{discrete_problem} admits a unique solution. Standard arguments yield \emph{a priori} error estimates in the usual DG norms; see \eg~\cite{EtangsaleCAMWA2022}.
\end{remark}

\begin{lemma}[Discrete trace approximations]
\label{lemma_Trace}
Let $\cdu\eqbydef(\du,\dtu)\in\CVh{}$ be the solution of~\eqref{discrete_problem} and let $(\dsg(\du),\dtsg(\cdu))$ be as in \cref{lemma-flux-reform}. Then, the numerical traces $(\dtu,\dtsg(\cdu)\cdot\vect n)$ are single-valued on $\Fh$ and satisfy
\begin{subequations}
\label{traces_def}
\begin{align}
  \label{traces_def_u}
  \dtu
  &=\begin{cases}
    \wmean{\du}{\vomega^{\tau}}-\varrho^{\tau}_{1}\jump{\bkap\grdh{\du}}&\text{on }\Fhi,\\
    0 &\text{on }\Fhb.
    \end{cases}\\
  \label{traces_def_sg}
  \dtsg(\cdu)
  &=
  -\cwmean{\bkap\grdh{\du}}{\vomega^{\tau}}
  +\varrho^{\tau}_{0}\jump{\du}
  \quad\text{on}\quad\Fh,
\end{align}
\end{subequations}
where $\vomega^{\tau}$ and $\varrho^{\tau}_{0,1}$ are $\tau$-dependent
coefficients defined as follows:
\begin{itemize}
  \item For $\f{}\eqbydef\bnd{\e{1}}\cap\bnd{\e{2}}\in\Fhi$:
  \[
    \vomega^{\tau}|_{\f{}}
    \eqbydef(\omega^{\tau}_1,\omega^{\tau}_2),
    \quad
    \varrho^{\tau}_{0}|_{\f{}}
    \eqbydef\frac{\tau_1\tau_2}{\tau_1+\tau_2},
    \quad
    \varrho^{\tau}_{1}|_{\f{}}
    \eqbydef\frac{1}{\tau_1+\tau_2},
  \]
  where $\tau_i\eqbydef\tau|_{\e{i},\f{}}$ and
  $\omega^{\tau}_i\eqbydef\tau_i/(\tau_1+\tau_2)$, $i=1,2$.
  \item For $\f{}\eqbydef\bnd{\e{}}\cap\bnddom\in\Fhb$:
  $\varrho^{\tau}_{0}|_{\f{}}\eqbydef\tau_{\e{},\f{}}$ (and $\varrho^{\tau}_{1}$
  is not needed).
\end{itemize}
In particular, \eqref{traces_def_sg} shows that $\dtsg(\cdu)$ can be expressed
solely in terms of $\du$; in the sequel we write $\dtsg(\du)$ for this reduced
expression.
\end{lemma}

\begin{proof}
The proof is immediate. Testing \eqref{discrete_problem} with $\cdv=(0,\dtv)$ and using the flux decomposition~\eqref{tracesg}, one obtains $\dualprod{\jump{\dtsg(\cdu)}}{\dtv}{0}{\Fh}=0$ for all $\dtv\in\TVh{}$, which implies $\jump{\dtsg(\cdu)}=0$ on $\Fhi$ by density. Relation~\eqref{traces_def_u} then follows immediately by isolating $\dtu$
from this conservation condition. The second relation~\eqref{traces_def_sg} is obtained by substituting \eqref{traces_def_u} into~\eqref{tracesg}.
\end{proof}
\begin{remark}[Natural bounds for $\varrho^{\tau}_{0,1}$]
\label{bound_dg_parameters}
Let $\f{}=\bnd{\e{1}}\cap\bnd{\e{2}}\in\Fhi$ and set $\tau_i\eqbydef\tau|_{\e{i},\f{}}$. Following \cref{lemma_Trace}, $\varrho^{\tau}_{0}|_{\f{}}$ and $\varrho^{\tau}_{1}|_{\f{}}$ are each one half of the harmonic mean and inverse arithmetic mean of $\tau_1$ and $\tau_2$, respectively, reflecting their natural duality. Hence,
\begin{subequations}\label{estimates_coefs}
\begin{align}
  \label{estimates_coef1}
  \varrho^{\tau}_{0}|_{\f{}}&\le \min(\tau_1,\tau_2),\\
  \label{estimates_coef2}
  \varrho^{\tau}_{1}|_{\f{}}&\le \min(\tau_1^{-1},\tau_2^{-1}).
\end{align}
\end{subequations}
These bounds will be used later in the stability and error analysis of \cref{S:4}.
\end{remark}

\subsection{Elimination of the skeletal unknown}

The key strategy consists in eliminating the discrete trace $\dtu$ from all skeletal contributions. To this end, we repeatedly use the identity relation~\eqref{magic3} from \cref{lemma:identity_rel}, together with the explicit expression of $\dtu$ and $\dtsg(\cdu)$ given in~\eqref{traces_def}. We first reduce the symmetry term, then the combined consistency--penalty term.
\begin{lemma}[Elimination of the skeletal unknown]
\label{lemma_sym_decomp}
Let $\cdu\eqbydef(\du,\dtu)\in\CVh{}$ be the solution of \eqref{discrete_problem}, and let the triplet $(\vomega^{\tau},\varrho^{\tau}_{0},\varrho^{\tau}_{1})$ be as in \cref{lemma_Trace}. Then, for all $\cdv\in\CVh{}$,
\begin{align}
  \label{symmetry_term}
  \dualprod{\dsg(\dv)}{\tjump{\cdu}}{0}{\bnd{\Th}}
  &=
  -\dualprod{\cwmean{\bkap\grdh{\dv}}{\vomega^{\tau}}}{\jump{\du}}{0}{\Fh}
  -\dualprod{\varrho^{\tau}_{1}\jump{\bkap\grdh{\du}}}{\jump{\bkap\grdh{\dv}}}{0}{\Fhi},
  \\
  \label{consistency_term}
  \dualprod{\dtsg(\cdu)}{\tjump{\cdv}}{0}{\bnd{\Th}}
  &=
  -\dualprod{\cwmean{\bkap\grdh{\du}}{\vomega^{\tau}}}{\jump{\dv}}{0}{\Fh}
  +\dualprod{\varrho^{\tau}_{0}\jump{\du}}{\jump{\dv}}{0}{\Fh}.
\end{align}
\end{lemma}
\begin{proof}
\textbf{Symmetry term.} We apply \eqref{magic3} with $\vect{b}=\dsg(\dv)$ and $\boldsymbol\varphi=(\du,\dtu)$. Since $\dtu=0$ on $\Fhb$,
\begin{equation}
  \label{symmetry_decomp}
  \dualprod{\dsg(\dv)}{\tjump{\cdu}}{0}{\bnd{\Th}}
  =
  \dualprod{\cwmean{\dsg(\dv)}{\vomega}}{\jump{\du}}{0}{\Fh}
  +\dualprod{\jump{\dsg(\dv)}}{\wmean{\du}{\vomega}-\dtu}{0}{\Fhi}.
\end{equation}
Using~\eqref{traces_def_u} and choosing $\vomega=\vomega^\tau$ gives $\wmean{\du}{\vomega^\tau}-\dtu=\varrho_1^\tau\jump{\bkap\grdh{\du}}$, hence
\[
  \dualprod{\jump{\dsg(\dv)}}{\wmean{\du}{\vomega^\tau}-\dtu}{0}{\Fhi}
  =
  -\dualprod{\varrho^{\tau}_{1}\jump{\bkap\grdh{\du}}}{\jump{\bkap\grdh{\dv}}}{0}{\Fhi}.
\]
Plugging into \eqref{symmetry_decomp} yields \eqref{symmetry_term}.\\
\smallskip
\textbf{Consistency--penalty term.} Similarly, we apply \eqref{magic3} with $\vect{b}=\dtsg(\cdu)$ and $\boldsymbol\varphi=(\dv,\dtv)$. Since $\dtv=0$ on $\Fhb$,
\begin{equation*}
  \dualprod{\dtsg(\cdu)}{\tjump{\cdv}}{0}{\bnd{\Th}}
  =
  \dualprod{\cwmean{\dtsg(\cdu)}{\vomega}}{\jump{\dv}}{0}{\Fh}
  +\dualprod{\jump{\dtsg(\cdu)}}{\wmean{\dv}{\vomega}-\dtv}{0}{\Fhi}.
\end{equation*}
The conservation condition $\jump{\dtsg(\cdu)}=0$ on $\Fhi$ makes the last term vanish. Finally, substituting~\eqref{traces_def_sg} yields \eqref{consistency_term}.
\end{proof}

\subsection{The Unified IP-DG formulation}

We can now state the main result of this section. Hereafter, we write $(\vomega,\varrho_{0},\varrho_{1})$ in place of $(\vomega^{\tau},\varrho_{0}^{\tau},\varrho_{1}^{\tau})$ for brevity. 
\begin{theorem}[Equivalence of H-IP and UIP-DG]
\label{thm_UIP_method}
Find $\du\in\Vh$ such that
\begin{equation}
  \label{UIP_DG_formulation}
  \dbilin{a_{h}^{(\epsilon)}}{\du}{\dv}
  =\inerprod{f}{\dv}{0}{\Th},
  \quad\forall\dv\in\Vh,
\end{equation}
where $\epsilon\in\{0,\pm1\}$. The discrete bilinear form
$a^{(\epsilon)}_h:\Vh\times\Vh\to\R$ decomposes as
\begin{equation*}
  \dbilin{a^{(\epsilon)}_h}{\du}{\dv}
  \eqbydef
  \inerprod{\bkap\grdh{\du}}{\grdh{\dv}}{0}{\Th}
  +\dbilin{a^{(c)}_h}{\du}{\dv}
  +\epsilon\dbilin{a^{(c)}_h}{\dv}{\du}
  +\dbilin{s^{(\epsilon)}_h}{\du}{\dv},
\end{equation*}
where the consistency, symmetry, and stabilization terms are defined by
\begin{align}
  \label{consistency_uip_term}
  \dbilin{a^{(c)}_h}{\du}{\dv}
  &\eqbydef
  -\dualprod{\cwmean{\bkap\grdh{\du}}{\vomega}}{\jump{\dv}}{0}{\Fh},\\
  \label{stability_uip_term}
  \dbilin{s^{(\epsilon)}_h}{\du}{\dv}
  &\eqbydef
  \dualprod{\varrho_{0}\jump{\du}}{\jump{\dv}}{0}{\Fh}
  -\epsilon\dualprod{\varrho_{1}\jump{\bkap\grdh{\du}}}
  {\jump{\bkap\grdh{\dv}}}{0}{\Fhi},
\end{align}
and the triplet $(\vomega,\varrho_{0},\varrho_{1})$ is uniquely determined
by $\tau$ through the hybrid formulation (see, \cref{lemma_Trace}). The DG formulation \eqref{UIP_DG_formulation} will be referred to as the \emph{Unified Interior Penalty Discontinuous Galerkin} (UIP-DG) method. Hence, the HIP and UIP-DG problems are equivalent in the sense that:
\begin{enumerate}
  \item[(i)] if $(\du^{\mathrm{HIP}},\dtu^{\mathrm{HIP}})\in\CVh{}$ solves
  \eqref{discrete_problem}, then $\du^{\mathrm{HIP}}\in\Vh$ is also the unique solution
  of~\eqref{UIP_DG_formulation} (Uniqueness of the UIP-DG solution is discussed in \cref{S:4.2}.);
  \item[(ii)] if $\du\in\Vh$ solves \eqref{UIP_DG_formulation} and $\dtu\in\TVh{}$
  is reconstructed from $\du$ using~\eqref{traces_def_u}, then $(\du,\dtu)$
  is the unique solution of~\eqref{discrete_problem}.
\end{enumerate}
\end{theorem}

\begin{proof}
(\emph{Forward implication})
Assertion \textnormal{(i)} follows directly from \cref{lemma_sym_decomp}: if
$(\du^{\mathrm{HIP}},\dtu^{\mathrm{HIP}})\in\CVh{}$ solves
\eqref{discrete_problem}, then substituting the skeletal contributions by the
expressions obtained in \cref{lemma_sym_decomp} yields precisely the primal
formulation \eqref{UIP_DG_formulation} satisfied by $\du^{\mathrm{HIP}}$.

\smallskip
\noindent
(\emph{Reverse implication})
Let $\du\in\Vh$ solve \eqref{UIP_DG_formulation}. We reconstruct $\dtu\in\TVh{}$ and the numerical flux $\dtsg(\du)$ from $\du$ by the explicit face-local formulas of \cref{lemma_Trace}. By construction, these formulas are equivalent, on each face, to the local H-IP relations written in terms of $\cdu\eqbydef(\du,\dtu)$; in particular,
\[
\dtsg(\cdu)=\dsg(\du)+\tau\tjump{(\du,\dtu)},
\qquad
\dsg(\du)\eqbydef -\bkap\grdh{\du},
\]
Let $\cdv=(\dv,\dtv)\in\CVh{}$. By linearity,
\[
\dbilin{\boldsymbol a_h^{(\epsilon)}}{(\du,\dtu)}{\cdv}
=
\dbilin{\boldsymbol a_h^{(\epsilon)}}{(\du,\dtu)}{(\dv,0)}
+
\dbilin{\boldsymbol a_h^{(\epsilon)}}{(\du,\dtu)}{(0,\dtv)}.
\]
Using now \cref{lemma_sym_decomp} and the fact that $\du$ solves
\eqref{UIP_DG_formulation}, the first term satisfies
\[
\dbilin{\boldsymbol a_h^{(\epsilon)}}{(\du,\dtu)}{(\dv,0)}
=
\dbilin{a_h^{(\epsilon)}}{\du}{\dv}
=
\inerprod{f}{\dv}{0}{\Th}.
\]
For the second term, using the flux form of the hybrid bilinear form and applying \eqref{magic3} with $\vect{b}=\dtsg(\cdu)$ and $\boldsymbol\varphi=(0,\dtv)$ we obtain
\[
\dbilin{\boldsymbol a_h^{(\epsilon)}}{(\du,\dtu)}{(0,\dtv)}
=
\dualprod{\dtsg(\cdu)}{\tjump{(0,\dtv)}}{0}{\bnd{\Th}}=-\dualprod{\jump{\dtsg(\du)}}{\dtv}{0}{\Fhi}=0,
\]
since $\jump{\dtsg(\du)}=0$ on $\Fhi$. Hence,
$\dbilin{\boldsymbol a_h^{(\epsilon)}}{(\du,\dtu)}{(0,\dtv)}=0$.
Therefore,
\[
\dbilin{\boldsymbol a_h^{(\epsilon)}}{(\du,\dtu)}{\cdv}
=
\inerprod{f}{\dv}{0}{\Th},
\quad \forall \cdv=(\dv,\dtv)\in\CVh{},
\]
so that $(\du,\dtu)$ is also a solution of \eqref{discrete_problem}. By uniqueness
of the H-IP problem, $(\du,\dtu)$ is its unique solution.
\end{proof}

The parameter $\epsilon\in\{0,\pm1\}$ yields three variants: IUIP-DG ($\epsilon=0$), SUIP-DG ($\epsilon=1$), and NUIP-DG ($\epsilon=-1$). We next provide an equivalent decomposition of the consistency term \eqref{consistency_uip_term} that clarifies the role of the skeletal weights.
\begin{lemma}[Decomposition of the consistency term]
\label{lemma:w_invariance}
For all $\dv,\dw\in\Vh$, the consistency term satisfies
\begin{equation}
  \label{w_decomp}
  \dbilin{a^{(c)}_h}{\dv}{\dw}
  =
  -\dualprod{\mean{\bkap\grdh{\dv}}}{\jump{\dw}}{0}{\Fh}
  +\dualprod{\vect{\gamma}\,\jump{\bkap\grdh{\dv}}}{\jump{\dw}}{0}{\Fhi},
\end{equation}
where $\vect{\gamma}\eqbydef\jump{\tau}/(4\mean{\tau})$ is a dimensionless vector field supported on $\Fhi$ and aligned with the interface normal.
\end{lemma}
\begin{proof}
The statement is immediate on $\Fhb$. Let $\f{}\in\Fhi$ and fix the orientation $\vect{n}_2=-\vect{n}_1$ so that $\jump{\dw}=(w_1-w_2)\vect{n}_1$. Only the normal component of $\bkap\grdh{\dv}$ contributes; setting $b_i\eqbydef(\bkap\grdh{\dv})_i\cdot\vect{n}_1$ for $i=1,2$,
\[
  \dualprod{\cwmean{\bkap\grdh{\dv}}{\vomega}}{\jump{\dw}}{0}{\f{}}
  =\dualprod{\cwmean{b}{\vomega}}{w_1-w_2}{0}{\f{}}
\]
Since $\omega_1+\omega_2=1$, a direct computation gives $\cwmean{b}{\vomega}=\omega_2 b_1+\omega_1 b_2=\mean{b}+\shalf(\omega_2-\omega_1)(b_1-b_2)$, whence
\[
  \dualprod{\cwmean{b}{\vomega}}{w_1-w_2}{0}{\f{}}
  =\dualprod{\mean{\bkap\grdh{\dv}}}{\jump{\dw}}{0}{\f{}}
  -\dualprod{\vect\gamma\jump{\bkap\grdh{\dv}}}{\jump{\dw}}{0}{\f{}},
\]
with $\vect\gamma:=\frac12(\omega_1-\omega_2)\vect{n}_{1}$. Using $\omega_i=\tau_i/(\tau_1+\tau_2)$ gives $\vect\gamma=\jump{\tau}/(4\mean{\tau})$.
\end{proof}

\begin{remark}[Structure of UIP-DG and links with existing schemes]
\label{rem:structure_uipdg}
All interface ingredients of UIP-DG are inherited from the H-IP stabilization $\tau$ through the triplet $(\vomega,\varrho_0,\varrho_1)$, where $\vomega$ governs the consistency/symmetry terms and $\varrho_{0,1}$ weights the primal and flux penalizations. \Cref{lemma:w_invariance} reveals that $a_h^{(c)}$ decomposes into the standard arithmetic-mean consistency term plus a flux-jump correction scaled by $\vect\gamma=\jump{\tau}/(4\mean{\tau})$. This correction vanishes when $\tau$ is \emph{single-valued} ($\tau_1=\tau_2$), in which case $\vomega$ reduces to arithmetic weights; even then, UIP-DG does not reduce to a classical IP-DG scheme since $\varrho_0=\tau/2$ and the flux-jump penalty weighted by $\varrho_1=1/(2\tau)$ persist. This degenerate setting corresponds to the analyses of~\cite{wells2011analysis,FabienNM2019}, which the present work extends to the practically relevant double-valued case. More generally, the transmissibility weights $\omega_i^\tau = \tau_i/(\tau_1+\tau_2)$, which by \cref{definition:penalty_parameter} read explicitly
\begin{equation*}
    \omega_i^\tau
    = \frac{\kappa_{\e{i},\f{}}\, h_{\e{i}}^{-1}}
           {\kappa_{\e{1},\f{}}\, h_{\e{1}}^{-1} + \kappa_{\e{2},\f{}}\, h_{\e{2}}^{-1}},
\end{equation*}
differ from the purely diffusivity-based weights $\omega_i^\kappa\eqbydef\kappa_{\e{i},\f{}}/(\kappa_{\e{1},\f{}}+\kappa_{\e{2},\f{}})$ of WIP-DG~\cite{BurmanSIAM2006,ESZIMA08} by the local mesh-size ratio $h_{\e{1}}/h_{\e{2}}$; the two coincide on uniform meshes ($h_{E_1}=h_{E_2}$). The flux-jump penalization $-\epsilon\,\varrho_1\jump{\bkap\grdh{\du}}
\cdot\jump{\bkap\grdh{\dv}}$ recovers a Douglas--Dupont-type stabilization~\cite{Douglas1976InteriorPP}, induced \emph{automatically} by the H-IP reduction. For $\epsilon=+1$, the negative sign --- here dictated by the adjoint-consistency structure inherited from~\cite{ABCMUnified} --- echoes an earlier observation of Codina, Principe and Baiges~\cite{CodinaCMAME2009}, who first identified, in the variational multiscale framework, that a flux-transmission condition between subscales naturally produces a sign-reversed penalization of the gradient jump. Although the structural origin differs (VMS continuous approximation versus H-IP reduction), the recurrence of the negative sign across two unrelated derivations is, in our view, not coincidental. The scalings $\varrho_0=\bigO{\tau}$ and $\varrho_1=\bigO{1/\tau}$ reflect the natural duality between primal and flux stabilizations, and will be used in \cref{S:4} to establish coercivity and boundedness of $a_h^{(\epsilon)}$.
\end{remark}

\begin{remark}[Extension to mixed boundary conditions]
\label{rem:bc_extensions}
The UIP-DG formulation~\eqref{UIP_DG_formulation} extends naturally to mixed boundary conditions. Let
\[
\partial\Omega=\Gamma_D\cup\Gamma_N,
\qquad
\Gamma_D\cap\Gamma_N=\emptyset,
\qquad
|\Gamma_D|>0,
\]
and consider
\[
u=g_D \quad\text{on }\Gamma_D,
\qquad
-\bkap\nabla u\cdot\vect n=g_N \quad\text{on }\Gamma_N,
\]
with $g_D\in H^{1/2}(\Gamma_D)$ and $g_N\in L^2(\Gamma_N)$. The boundary skeleton is partitioned as
\[
\Fhb=\FhD\cup\FhN,
\]
where $\FhD$ and $\FhN$ denote the sets of boundary faces lying on $\Gamma_D$ and $\Gamma_N$, respectively. All interior interface contributions in $a_h^{(\epsilon)}$ remain unchanged. On boundary faces, only the terms associated with the weak enforcement of the Dirichlet condition are kept, and they are restricted to $\FhD$. In particular, the boundary consistency, adjoint-consistency, and penalty terms appearing in the UIP-DG bilinear form are integrated on $\FhD$ only. On $\FhN$, no boundary penalty term is added, and the Neumann condition enters solely through the prescribed normal diffusive flux in the linear functional. The mixed-boundary UIP-DG problem therefore reads: find $\du\in\Vh$ such that, for all $\dv\in\Vh$,
\[
\dbilin{a^{(\epsilon)}_h}{\du}{\dv}=\ell^{(\epsilon)}_h(\dv),
\]
where $a_h^{(\epsilon)}$ is obtained from the bilinear form of \cref{thm_UIP_method} by restricting all boundary Dirichlet terms from $\Fhb$ to $\FhD$, and $\ell_h^{(\epsilon)}:\Vh\to\R$ is defined by
\[
\ell^{(\epsilon)}_h(\dv)
:= \inerprod{f}{\dv}{0}{\Th}
- \dualprod{g_N}{\dv}{0}{\FhN}
+ \epsilon\,\dualprod{\bkap\grdh{\dv}\cdot\vect n}{g_D}{0}{\FhD}
+ \dualprod{\varrho_0\,g_D}{\dv}{0}{\FhD}.
\]
Here, for any Dirichlet boundary face $\f{}\in\FhD$, the boundary penalty parameter is still defined by
\[
\varrho_0|_{\f{}}\eqbydef\tau_{\e{},\f{}},
\qquad \f{}\in\FhD.
\]
The consistency and stability arguments of \cref{S:4} extend to this setting with the usual modifications, the only additional contributions arising from the approximation of the Dirichlet datum on $\Gamma_D$ and from the natural Neumann boundary term on $\Gamma_N$.
\end{remark}

\section{A priori error analysis}
\label{S:4}

This section establishes \emph{a priori} error estimates for the UIP-DG method. The analysis is based on a Strang-type argument for DG formulations \cite{di2011mathematical,Riviere08bookDG} and relies on consistency, discrete coercivity, and boundedness of the UIP-DG bilinear form with respect to mesh-dependent energy norms. 

\subsection{Preliminaries: norms and trace tools}
\label{S:4.1}

We introduce the mesh-dependent norms used in the stability and error analysis and recall the trace inequalities required to control interface terms. Throughout this section, $\C{}$ denotes a generic positive constant independent of $h$, possibly depending on $k$, $d$, and the mesh shape-regularity parameter.

\paragraph{Trace inequalities}
Let $\Th$ be a shape-regular simplicial mesh and let $\f{}\in\FE$. The
\emph{continuous trace} inequality states that there exists $\C{tr}>0$ such that,
for all $v\in\hilbert{1}{\e{}}{}$,
\begin{equation}
  \label{eq:cont_trace_4.1}
  \norm{v}{0}{\f{}}^2
  \le\C{tr}\Bigl(h_{\e{}}^{-1}\norm{v}{0}{\e{}}^2
  + h_{\e{}}\norm{\grd{v}}{0}{\e{}}^2\Bigr).
\end{equation}
Since $\Vh|_{\e{}}\subset\Pk{k}{\e{}}$, combining \eqref{eq:cont_trace_4.1} with a standard inverse estimate \cite{Brenner2008} yields the \emph{discrete trace} inequality: there exists $\C{T}>0$ such that, for all $w\in\Pk{m}{\e{}}$ with $m\le k$,
\begin{equation}
  \label{discret_trace_ineq}
  \norm{w}{0}{\f{}}^2
  \le \C{T}^2\,h_{\e{}}^{-1}\norm{w}{0}{\e{}}^2.
\end{equation}

\paragraph{Energy norms}
For the stability analysis, we introduce the following energy norm, \ie for all $\dv\in\Vh$, 
\begin{equation}
  \label{nrj_norm1}
  \tnorm{\dv}^2
  \eqbydef
  \norm{\bkap^{\ohalf}\grdh{\dv}}{0}{\Th}^2
  +\seminorm{\dv}{0}{h}^2,
  \qquad
  \seminorm{\dv}{0}{h}^2
  \eqbydef\sumF\norm{\varrho_{0}|^{\ohalf}_{\f{}}\jump{\dv}}{0}{\f{}}^2.
\end{equation}
The seminorm $|\cdot|_{0,h}$ controls the penalization of interelement jumps
induced by the UIP numerical fluxes through the coefficient $\varrho_{0}$.
For the boundedness analysis, we also introduce an augmented norm on $V(h)\eqbydef V+\Vh$, defined for all $v\in V(h)$ by
\begin{equation}
  \label{nrj_norm2}
  \tnorm{v}_{*}^2
  \eqbydef
  \tnorm{v}^2+\seminorm{v}{1}{h}^2,
  \qquad
  \seminorm{v}{1}{h}^2
  \eqbydef\sumT h_{\e{}}
  \norm{\bkap^{\ohalf}\grdh{v}}{0}{\FE}^2.
\end{equation}
The seminorm $\seminorm{\cdot}{1}{h}$ is used to control diffusive interface traces arising in the consistency and boundedness estimates. By the discrete trace inequality~\eqref{discret_trace_ineq}, we infer that, for all $\dv\in\Vh$,
\begin{equation}
\label{eq:discrete-trace-control}
\seminorm{\dv}{1}{h}
\le
\C{dt}\norm{\bkap^{\ohalf}\grdh{\dv}}{0}{\Th},
\end{equation}
where $\C{dt}\eqbydef \eta_0^{\ohalf}\C{T}$ depends only on the mesh regularity parameters.

\paragraph{Useful interface bounds}
The next lemma and corollary gather the key interface estimates required to control the skeleton contributions in the coercivity and continuity analyses. 

\begin{lemma}[Extended Interface bounds]
\label{lem:extended-interface-bounds}
Let $\C{1}\eqbydef(\alpha_0 C_T^2)^{-\ohalf}>0$ be a constant independent of $h$ and of the diffusion contrast but depending only on the user-constant $\alpha_0$ in \eqref{penalty_term}. Then the following bounds hold;
\begin{enumerate}
  \item[\textnormal{(i)}] \textbf{Weighted consistency bound.}
  For all $(v,\dw)\in V(h)\times \Vh$,
  \begin{equation}
      \label{eq:extended-weighted-consistency}
      \abs{\dualprod{\cwmean{\bkap\grdh{v}}{\vomega}}
      {\jump{\dw}}{0}{\Fh}}
      \le
      \C{1}\seminorm{v}{1}{h}\,\seminorm{\dw}{0}{h},
  \end{equation}
  \item[\textnormal{(ii)}] \textbf{Flux-jump bound.}
  For all $v\in V(h)$,
  \begin{equation}
    \label{eq:extended-flux-jump}
    \norm{\varrho_1^{\ohalf}\jump{\bkap\grdh{v}}}{0}{\Fhi}
    \leq
    \C{1}\,\seminorm{v}{1}{h}.
  \end{equation}
\end{enumerate}
\end{lemma}
\begin{proof}[Proof of (i)]
We argue face by face. For $\f{}\in\Fhi$ shared by $\e{1}$ and $\e{2}$, set
$q_i\eqbydef\bkap_i\grdh{v_i}\cdot\vect{n}_i$, so that
$(\cwmean{\bkap\grdh{v}}{\vomega})\cdot\vect{n}_1=\omega_2 q_1-\omega_1 q_2$.
By Cauchy--Schwarz and the structural identity $\omega_2^2\tau_1+\omega_1^2\tau_2=\varrho_0$,
\[
\abs{T_{\f{}}}
\le
\Bigl(\tau_1^{-1}\norm{q_1}{0}{\f{}}^2+\tau_2^{-1}\norm{q_2}{0}{\f{}}^2\Bigr)^{\ohalf}
\norm{\varrho_0^{\ohalf}\jump{\dw}}{0}{\f{}}.
\]
Since $\norm{q_i}{0}{\f{}}^2\le\kappa_i\norm{\bkap_i^{\ohalf}\grdh{v_i}}{0}{\f{}}^2$
and $\tau_i=\alpha_0\C{T}^2\kappa_i h_i^{-1}$, one has
$\tau_i^{-1}\norm{q_i}{0}{\f{}}^2\le\C{1}^2 h_i\norm{\bkap_i^{\ohalf}\grdh{v_i}}{0}{\f{}}^2$
with $\C{1}\eqbydef(\alpha_0\C{T}^2)^{-\ohalf}$.
The same estimate holds on $\f{}\in\Fhb$ with a single term (there
$\cwmean{\bkap\grdh{v}}{\vomega}=\bkap\grdh{v}$ and $\varrho_0=\tau$).
Summing over $\Fh$ and applying Cauchy--Schwarz yields~\eqref{eq:extended-weighted-consistency}.
\end{proof}
\begin{proof}[Proof of (ii)]
Let $\f{}\in\Fhi$ be shared by $\e{1}$ and $\e{2}$, and set $q_i\eqbydef \bkap_i\grdh{v_i}\cdot\vect n_i$ for $i\in\{1,2\}$. Using the weighted Cauchy--Schwarz inequality gives
\[
\norm{q_1+q_2}{0}{\f{}}^2
\le
(\tau_1+\tau_2)\bigl(\tau_1^{-1}\norm{q_1}{0}{\f{}}+\tau_2^{-1}\norm{q_2}{0}{\f{}}\bigr).
\]
Since $\varrho_1=(\tau_1+\tau_2)^{-1}$ on $\Fhi$, we infer
$\varrho_1\,\norm{q_1+q_2}{0}{\f{}}^2\le
\tau_1^{-1}\norm{q_1}{0}{\f{}}^2+\tau_2^{-1}\norm{q_2}{0}{\f{}}^2
$. Therefore,
\[
\norm{\varrho_1^{\ohalf}\jump{\bkap\grdh{v}}}{0}{\f{}}^2
\le
\tau_1^{-1}\norm{q_1}{0}{\f{}}^2+\tau_2^{-1}\norm{q_2}{0}{\f{}}^2.
\]
Proceeding exactly as above, we infer
\[
\norm{\varrho_1^{\ohalf}\jump{\bkap\grdh{v}}}{0}{\f{}}^2
\le
\C{1}
\Bigl(
h_1\norm{\bkap_1^{\ohalf}\grdh{v_1}}{0}{\f{}}^2
+
h_2\norm{\bkap_2^{\ohalf}\grdh{v_2}}{0}{\f{}}^2
\Bigr),
\]
where $\C{1}\eqbydef(\alpha_0 C_T^2)^{-\ohalf}$. Summing over $\Fhi$ yields \eqref{eq:extended-flux-jump}.
\end{proof}

\begin{corollary}[Discrete interface bounds]
\label{cor:discrete-interface-bounds}
Let $\C{0}\eqbydef(\eta_0/\alpha_0)^{\ohalf}>0$, independent of $h$ and the diffusion contrast. For all $\dv\in\Vh$,
\begin{subequations}
\label{eq:discrete-interface-bounds}
\begin{align}
\label{eq:discrete-weighted-consistency}
\abs{\dualprod{\cwmean{\bkap\grdh{\dv}}{\vomega}}
{\jump{\dv}}{0}{\Fh}}
&\leq
\C{0}\norm{\bkap^{\ohalf}\grdh{\dv}}{0}{\Th}\,\seminorm{\dv}{0}{h},
\\
\label{eq:discrete-flux-jump}
\norm{\varrho_1^{\ohalf}\jump{\bkap\grdh{\dv}}}{0}{\Fhi}
&\leq
\C{0}
\norm{\bkap^{\ohalf}\grdh{\dv}}{0}{\Th}.
\end{align}
\end{subequations}
These follow from \cref{lem:extended-interface-bounds} with $v=\dv$ and~\eqref{eq:discrete-trace-control}, using $\C{0}=\C{dt}\C{1}$.
\end{corollary}

\subsection{Stability analysis}
\label{S:4.2}

\begin{theorem}[Discrete coercivity]
\label{thm:coercivity}
Let $\epsilon\in\{0,\pm1\}$. Assume that the structural constant $\C{0}$ introduced in \cref{cor:discrete-interface-bounds} satisfies $\C{0}^2<\ohalf$ (or equivalently $\alpha_0>2\,\eta_0$). Then there exists a constant $\C{sta}>0$, depending only on $\C{0}$,
but independent of $h$, the diffusion contrast, and the discrete
trace constant $\C{T}$, such that
\begin{equation}
\label{eq:coercivity}
a_h^{(\epsilon)}(\dv,\dv)\ge \C{sta}\,\tnorm{\dv}^2
\qquad\forall\,\dv\in\Vh.
\end{equation}
\end{theorem}

\begin{proof}
Setting $\du=\dv$ in \eqref{UIP_DG_formulation}, we write
\[
a_h^{(\epsilon)}(\dv,\dv)=\mathfrak{T}_1-(1+\epsilon)\mathfrak{T}_2-\epsilon \mathfrak{T}_3,
\]
with
\begin{align*}
\mathfrak{T}_1 &\eqbydef
\norm{\bkap^{\ohalf}\grdh{\dv}}{0}{\Th}^2
+\norm{\varrho_0^{\ohalf}\jump{\dv}}{0}{\Fh}^2,\\
\mathfrak{T}_2 &\eqbydef
\dualprod{\cwmean{\bkap\grdh{\dv}}{\vomega}}{\jump{\dv}}{0}{\Fh},\\
\mathfrak{T}_3 &\eqbydef
\norm{\varrho_1^{\ohalf}\jump{\bkap\grdh{\dv}}}{0}{\Fhi}^2.
\end{align*}
Setting $X\eqbydef \norm{\bkap^{\ohalf}\grdh{\dv}}{0}{\Th}$ and $Y\eqbydef \seminorm{\dv}{0}{h}$,
we have $\mathfrak{T}_1=X^2+Y^2$.
By \eqref{eq:discrete-weighted-consistency}, $|\mathfrak{T}_2|\le\C{0}\,XY$, so Young's inequality
with parameter $\theta>0$ gives $|(1+\epsilon)\mathfrak{T}_2|\le\theta X^2+\frac{(1+\epsilon)^2\C{0}^2}{4\theta}Y^2$.
By \eqref{eq:discrete-flux-jump}, $|\epsilon\mathfrak{T}_3|\le|\epsilon|\C{0}^2 X^2$.
Combining these three bounds,
\[
a_h^{(\epsilon)}(\dv,\dv)
\ge
\bigl(1-\theta-\abs{\epsilon}\C{0}^2\bigr)X^2
+\left(1-\frac{(1+\epsilon)^2\C{0}^2}{4\theta}\right)Y^2.
\]
Since $|\epsilon|\le 1$ and $(1+\epsilon)^2\le 4$ for all $\epsilon\in\{0,\pm1\}$,
the most restrictive case is $\epsilon=1$. Choosing $\theta$ such that
\[
\C{0}^2<\theta<1-\C{0}^2,
\]
both coefficients of $X^2$ and $Y^2$ are positive. Such a choice is
possible if and only if $\C{0}^2<\frac12$ or equivalently $\alpha_0>2\eta_0$. Under this assumption, there exists $\C{sta}>0$, depending only on $\C{0}$, such that for all $\epsilon\in\{0,\pm1\}$,
\[
a_h^{(\epsilon)}(\dv,\dv)
\ge
\C{sta}\bigl(X^2+Y^2\bigr)
=
\C{sta}\,\tnorm{\dv}^2.
\]
This concludes the proof.
\end{proof}

\subsection{Consistency and boundedness}
\label{S:4.3}

\begin{lemma}[Consistency and Galerkin orthogonality]
\label{lem:consistency}
Let $u\in V\cap\hilbert{2}{\Th}{}$ be the solution
of~\eqref{weak_model_problem}. Then
\begin{equation}
  \label{eq:consistency}
  a_h^{(\epsilon)}(u,\dv)=\inerprod{f}{\dv}{0}{\Th}
  \qquad\forall\,\dv\in\Vh.
\end{equation}
Consequently, if $\du\in\Vh$ is the solution of~\eqref{UIP_DG_formulation},
then
\begin{equation}
  \label{eq:galerkin_orth}
  a_h^{(\epsilon)}(u-\du,\dv)=0
  \qquad\forall\,\dv\in\Vh.
\end{equation}
\end{lemma}

\begin{proof}
Since $u\in V=H^1_0(\Omega)$, we have $\jump{u}=\vect{0}$ on $\Fh$ and, as $u$
solves~\eqref{model_problem_strong}, $\jump{\bkap\grd{u}}=0$ on $\Fhi$.
All stabilization terms therefore vanish, and $\cwmean{\bkap\grd{u}}{\vomega}$
reduces to the common normal flux on each $\f{}\in\Fhi$. An elementwise
integration by parts then gives
\begin{align*}
a_h^{(\epsilon)}(u,\dv)
&=
\sumT
\left[
\inerprod{\bkap\grd{u}}{\grdh{\dv}}{0}{\e{}}
-
\dualprod{(\bkap\grd{u})\cdot\vect{n}_{\bnd{\e{}}}}{\dv}{0}{\FE}
\right]\\
&=
\sumT \inerprod{-\dvg(\bkap\grd{u})}{\dv}{0}{\e{}}=\inerprod{f}{\dv}{0}{\Th},
\end{align*}
which proves~\eqref{eq:consistency}. Subtracting~\eqref{UIP_DG_formulation}
from~\eqref{eq:consistency} yields~\eqref{eq:galerkin_orth}.
\end{proof}

\begin{lemma}[Boundedness]
\label{lem:boundedness}
Let $\epsilon\in\{0,\pm1\}$. There exists a constant $\C{bnd}>0$, depending only on $\alpha_0$,
$\eta_0$, and the mesh shape-regularity, but independent of $h$ and of
the diffusion contrast, such that
\begin{equation}
\label{eq:boundedness}
\abs{a_h^{(\epsilon)}(v,\dw)}
\le
\C{bnd}\,\tnorm{v}_{*}\tnorm{\dw}
\qquad
\forall\,(v,\dw)\in V(h)\times\Vh.
\end{equation}
\end{lemma}

\begin{proof}
For all $(v,\dw)\in V(h)\times\Vh$ and all $\epsilon\in\{0,\pm 1\}$,
\[
  \abs{\dbilin{a^{(\epsilon)}_h}{v}{\dw}}
  \leq
  \abs{\mathfrak{T}_{1}}+\abs{\mathfrak{T}_{2}}+\abs{\mathfrak{T}_{3}}
  +\abs{\mathfrak{T}_{4}}+\abs{\mathfrak{T}_5},
\]
where
\begin{align*}
  \mathfrak{T}_{1}
  &\eqbydef\inerprod{\bkap^{\ohalf}\grdh{v}}{\bkap^{\ohalf}\grdh{\dw}}{0}{\Th},\\
  \mathfrak{T}_{2}
  &\eqbydef\dualprod{\cwmean{\bkap\grdh{v}}{\vomega}}{\jump{\dw}}{0}{\Fh},\\
  \mathfrak{T}_{3}
  &\eqbydef\dualprod{\cwmean{\bkap\grdh{\dw}}{\vomega}}{\jump{v}}{0}{\Fh},\\
  \mathfrak{T}_{4}
  &\eqbydef\dualprod{\varrho_{0}^{\ohalf}\jump{v}}{\varrho_{0}^{\ohalf}\jump{\dw}}{0}{\Fh},\\
  \mathfrak{T}_{5}
  &\eqbydef\dualprod{\varrho_{1}^{\ohalf}\jump{\bkap\grdh{v}}}
  {\varrho_{1}^{\ohalf}\jump{\bkap\grdh{\dw}}}{0}{\Fhi}.
\end{align*}
We bound each term separately. The Cauchy--Schwarz inequality gives directly 
\begin{align*}
  \abs{\mathfrak{T}_1}
  &\leq\norm{\bkap^{\ohalf}\grdh{v}}{0}{\Th}\norm{\bkap^{\ohalf}\grdh{\dw}}{0}{\Th}\leq\tnorm{v}_{*}\tnorm{\dw},\\
  \abs{\mathfrak{T}_{4}}
  &\leq\seminorm{v}{0}{h}\seminorm{\dw}{0}{h}\leq\tnorm{v}_{*}\tnorm{\dw},\\
  \abs{\mathfrak{T}_{5}}
  &\leq(\C{1}\seminorm{v}{1}{h})(\C{0}\norm{\bkap^{\ohalf}\grdh{\dw}}{0}{\Th})
  \leq\C{0}\C{1}\tnorm{v}_{*}\tnorm{\dw}.
\end{align*}
For $\mathfrak{T}_2$ and $\mathfrak{T}_3$, we use directly \cref{lem:extended-interface-bounds}
\begin{align*}
  \abs{\mathfrak{T}_2}
  &\le\C{1}\seminorm{v}{1}{h}\seminorm{\dw}{0}{h}
  \leq\C{1}\tnorm{v}_{*}\tnorm{\dw},\\
  \abs{\mathfrak{T}_{3}}
  &\leq\C{0}\norm{\bkap^{\ohalf}\grdh{\dw}}{0}{\Th}\seminorm{v}{0}{h}
  \leq\C{0}\tnorm{v}_{*}\tnorm{\dw}.
\end{align*}
Collecting the above bounds yields \eqref{eq:boundedness} with
$\C{bnd}\eqbydef 2+\C{0}+\C{1}+\C{0}\C{1}$.
\end{proof}

\begin{remark}[Contrast-independent stability]
\label{rem:contrast_indep}
Under the condition $\alpha_0>2\eta_0$, the coercivity constant $\C{sta}$ depends only on $\alpha_0$ and $\eta_0$, whereas the boundedness constant $\C{bnd}$ depends on $\alpha_0$, $\eta_0$, and the mesh shape-regularity through $\C{1}=(\alpha_0\C{T}^2)^{-1/2}$. Both constants are uniform with respect to the mesh size $h$ and the diffusion contrast $\kappa_{\max}/\kappa_{\min}$. This uniformity is a direct consequence of the transmissibility-based scaling of the penalty parameter $\tau$ in \cref{definition:penalty_parameter}, which incorporates the local diffusivity into the stabilization and yields balanced interface control across heterogeneous media.
\end{remark}

\subsection{A priori error estimates}
\label{S:4.4}

We recall the following standard approximation properties of the elementwise $L^2$-projection $\proj{v}\in\Vh$, defined for each $\e{}\in\Th$ by
\[
\inerprod{\proj{v}}{w_h}{0}{\e{}}
=
\inerprod{v}{w_h}{0}{\e{}}
\qquad\forall w_h\in\Pk{k}{\e{}}.
\]
For all $v\in V\cap\hilbert{s}{\Th}{}$ with $1\leq s\leq k+1$, there exists $\C{}>0$, depending only on the polynomial degree $k$ and the mesh shape-regularity parameter $\eta_0$, such that
\begin{subequations}
  \label{eq:interp_estimates}
  \begin{align}
    \label{eq:interp_L2}
    \norm{v-\proj{v}}{0}{\e{}}
    &\leq
    \C{}\,h_{\e{}}^{s}\,\abs{v}_{s,\e{}},\\
    \label{eq:interp_grad}
    \norm{\grd{(v-\proj{v})}}{0}{\e{}}
    &\leq
    \C{}\,h_{\e{}}^{s-1}\,\abs{v}_{s,\e{}},\\
    \label{eq:interp_trace}
    \norm{v-\proj{v}}{0}{\f{}}
    &\leq
    \C{}\,h_{\e{}}^{s-\ohalf}\,\abs{v}_{s,\e{}},
    \qquad\f{}\in\FE.
  \end{align}
The constant $\C{}$ is independent of $h$ and of $\bkap$.
\end{subequations}
\begin{lemma}[Approximation in the augmented norm]
\label{lem:approx-star}
Let $u\in V\cap\hilbert{\mu}{\Th}{}$ with $2\le \mu\le k+1$, and let $\proj{u}\in\Vh$ denote the elementwise $\LL{2}$-projection of $u$ onto $\Vh$. Then there exists a constant $\C{int}>0$, independent of $h$ and of the diffusion contrast, such that
\begin{equation}
\label{eq:approx-star}
\tnorm{u-\proj{u}}_{*}
\le
\C{int}\norm{\bkap}{\infty}{\domain}^{\ohalf}
h^{\mu-1}\seminorm{u}{\mu}{\Th}.
\end{equation}
\end{lemma}

\begin{proof}
Set $\erp{u}{\pi}\eqbydef u-\proj{u}$. We bound each component of $\tnorm{\erp{u}{\pi}}_{*}^2 = \norm{\bkap^{\ohalf}\grdh{\erp{u}{\pi}}}{0}{\Th}^2+\seminorm{\erp{u}{\pi}}{0}{h}^2+\seminorm{\erp{u}{\pi}}{1}{h}^2$ separately. By~\eqref{eq:interp_grad} and $\bkap\le\norm{\bkap}{\infty}{\domain}\mathbf{I}$,
\[
\norm{\bkap^{\ohalf}\grdh{\erp{u}{\pi}}}{0}{\Th}^2
\le
\norm{\bkap}{\infty}{\domain}
\sumT \norm{\grdh{\erp{u}{\pi}}}{0}{\e{}}^2
\le
\C{}\norm{\bkap}{\infty}{\domain}
h^{2\mu-2}\seminorm{u}{\mu}{\Th}^2.
\]
Now since $u\in V$ implies $\jump{u}=\vect{0}$ on $\Fh$, one has $\jump{\erp{u}{\pi}}=-\jump{\proj{u}}$.
By the bound $\varrho_0\le\alpha_0\C{T}^2\kappa_{\e{}}\,h_{\e{}}^{-1}$
(from \eqref{estimates_coef1} and \cref{definition:penalty_parameter})
and the trace estimate~\eqref{eq:interp_trace},
\[
\seminorm{\erp{u}{\pi}}{0}{h}^2
\le
\C{}\norm{\bkap}{\infty}{\domain}
\sumT h_{\e{}}^{-1}\sum_{\f{}\in\FE}\norm{\erp{u}{\pi}}{0}{\f{}}^2
\le
\C{}\norm{\bkap}{\infty}{\domain}
h^{2\mu-2}\seminorm{u}{\mu}{\Th}^2.
\]
Hence applying the continuous trace inequality~\eqref{eq:cont_trace_4.1} to $\grdh{\erp{u}{\pi}}$
on each $\e{}\in\Th$ and using~\eqref{eq:interp_grad} together with the property $\seminorm{\grdh{\erp{u}{\pi}}}{1}{\e{}}\le\C{}h_{\e{}}^{\mu-2}\seminorm{u}{\mu}{\e{}}$,
\[
\seminorm{\erp{u}{\pi}}{1}{h}^2
=
\sumT h_{\e{}}\norm{\bkap^{\ohalf}\grdh{\erp{u}{\pi}}}{0}{\FE}^2
\le
\C{}\norm{\bkap}{\infty}{\domain}
h^{2\mu-2}\seminorm{u}{\mu}{\Th}^2.
\]
Summing the three bounds and taking square roots yields~\eqref{eq:approx-star}.
\end{proof}
\begin{theorem}[$\tnorm{\cdot}$-norm error estimate]
\label{thm:nrj-estimate}
Let $u\in V\cap\hilbert{2}{\Th}{}$ be the exact solution
of~\eqref{weak_model_problem} and $\du\in\Vh$ the discrete solution
of~\eqref{UIP_DG_formulation}. Under the assumptions of
\cref{thm:coercivity,lem:boundedness}, there exists $\C{qo}>0$,
depending only on the coercivity and boundedness constants and hence
independent of $h$ and the diffusion contrast, such that
\begin{equation}
  \label{infcond}
  \tnorm{u-\du}
  \leq
  \C{qo}\inf_{\dv\in\Vh}\tnorm{u-\dv}_{*}.
\end{equation}
If, moreover, $u\in\hilbert{\mu}{\Th}{}$ for some $2\le\mu\le k+1$,
then
\begin{equation}
  \label{nrj-estimate}
  \tnorm{u-\du}
  \leq
  \C{err}\norm{\bkap}{\infty}{\domain}^{\ohalf}
  h^{\mu-1}\seminorm{u}{\mu}{\Th},
\end{equation}
where $\C{err}>0$ is independent of $h$ and of the diffusion contrast.
In particular, for $\mu=k+1$, the estimate~\eqref{nrj-estimate} is
quasi-optimal of order $h^k$.
\end{theorem}

\begin{proof}
Let $\dv\in\Vh$. By the triangle inequality and the definition of the augmented norm,
\[
\tnorm{u-\du}
\le
\tnorm{u-\dv}+\tnorm{\du-\dv}
\le
\tnorm{u-\dv}_{*}+\tnorm{\du-\dv}.
\]
It remains to estimate $\tnorm{\du-\dv}$. By coercivity,
\[
\C{sta}\tnorm{\du-\dv}^2
\le
a_h^{(\epsilon)}(\du-\dv,\du-\dv).
\]
Using the Galerkin orthogonality \eqref{eq:galerkin_orth}, we obtain
\[
a_h^{(\epsilon)}(\du-\dv,\du-\dv)
=
a_h^{(\epsilon)}(u-\dv,\du-\dv).
\]
The boundedness estimate then yields
\[
\C{sta}\tnorm{\du-\dv}^2
\le
\C{bnd}\tnorm{u-\dv}_{*}\tnorm{\du-\dv}.
\]
Hence,
\[
\tnorm{\du-\dv}
\le
\C{sta}^{-1}\C{bnd}\tnorm{u-\dv}_{*}.
\]
Substituting into the triangle inequality gives
\[
\tnorm{u-\du}
\le
\left(1+\C{sta}^{-1}\C{bnd}\right)\tnorm{u-\dv}_{*}.
\]
Since $\dv\in\Vh$ is arbitrary, \eqref{infcond} follows with $\C{qo}\eqbydef 1+\C{sta}^{-1}\C{bnd}$. Choosing now $\dv\eqbydef\proj{u}$ and applying \cref{lem:approx-star} with $\mu=k+1$, we obtain
\[
\tnorm{u-\proj{u}}_{*}
\le
\C{int}\norm{\bkap}{\infty}{\domain}^{\ohalf}
h^k\seminorm{u}{k+1}{\Th}.
\]
Combining this estimate with \eqref{infcond} yields
\[
\tnorm{u-\du}
\le
\C{qo}\tnorm{u-\proj{u}}_{*}
\le
\C{qo}\C{int}\norm{\bkap}{\infty}{\domain}^{\ohalf}
h^k\seminorm{u}{k+1}{\Th},
\]
which is exactly \eqref{nrj-estimate} with $\C{err}\eqbydef\C{qo}\C{int}$.
\end{proof}

\begin{theorem}[$\LL{2}$-error estimates]
\label{thm:L2-estimates}
Under the assumptions of \cref{thm:nrj-estimate}, there exists $\C{}>0$, independent of $h$, such that
\begin{equation}
  \label{L2-estimate}
  \norm{u-\du}{0}{\Th}
  \leq
  \C{}\,h^{s_\epsilon}\,\tnorm{u-\du},
\end{equation}
where $s_\epsilon\eqbydef 1$ for the symmetric scheme (SUIP-DG, $\epsilon=1$) and $s_\epsilon\eqbydef 0$ otherwise (IUIP-DG and NUIP-DG, $\epsilon\in\{0,-1\}$). The optimal estimate ($s_\epsilon=1$) additionally requires elliptic regularity for the dual problem and relies on an Aubin--Nitsche argument. The suboptimal estimate ($s_\epsilon=0$) follows from a broken Poincar\'e inequality.
\end{theorem}

\begin{proof}

\textit{Suboptimal bound ($\epsilon\in\{0,-1\}$).}
Let $e\eqbydef u-\du\in V(h)$. Since $\jump{u}$ is null on $\Fh$, we have
$\jump{e}=-\jump{\du}$ and $\seminorm{e}{0}{h}=\seminorm{\du}{0}{h}$.
We now invoke the broken Poincar\'e--Friedrichs inequality~\cite{Brenner2003PF}:
there exists $\C{P}>0$, depending only on $\domain$ and the mesh
shape-regularity, such that
\[
  \norm{v}{0}{\Th}^2
  \leq
  \C{P}^2
  \left(
    \norm{\grdh{v}}{0}{\Th}^2
    + \sumF h_{\f{}}^{-1}\norm{\jump{v}}{0}{\f{}}^2
  \right),
  \qquad \forall\,v\in V(h).
\]
Using $\norm{\grdh{v}}{0}{\Th}^2\le \kappa_{\min}^{-1}\norm{\bkap^{\ohalf}\grdh{v}}{0}{\Th}^2$ and the lower bound on the penalty parameter
\[
  \varrho_0|_{\f{}}
  \ge \alpha_0 \C{T}^2 \kappa_{\min}\, h_{\f{}}^{-1},
  \qquad \forall \f{}\in\Fh,
\]
we infer that
\[
  \sumF h_{\f{}}^{-1}\norm{\jump{v}}{0}{\f{}}^2
  \le
  (\alpha_0 \C{T}^2 \kappa_{\min})^{-1}
  \seminorm{v}{0}{h}^2.
\]
Hence, the broken Poincar\'e--Friedrichs inequality yields
\[
  \norm{v}{0}{\Th}^2
  \le
  \frac{\C{P}^2}{\kappa_{\min}}
  \left(
    \norm{\bkap^{\ohalf}\grdh{v}}{0}{\Th}^2
    +
    (\alpha_0 \C{T}^2)^{-1}\seminorm{v}{0}{h}^2
  \right).
\]
Using the explicit discrete trace constant on simplices, for which $C_T^2 \ge 1$, and the coercivity condition $\alpha_0>2\eta_0=2(d+1)$, we have $\alpha_0 C_T^2>1$. Hence, the jump seminorm in $\tnorm{\cdot}$ controls the second term on the right-hand side with a coefficient strictly smaller than one. Therefore,
\[
  \norm{v}{0}{\Th}^2
  \le
  \frac{\C{P}^2}{\kappa_{\min}}\,\tnorm{v}^2,
  \qquad \forall v\in V(h).
\]
Applying this estimate to $v=e$ proves \eqref{L2-estimate} with $\C{}=\C{P}\,\kappa_{\min}^{-\ohalf}$ and $s_\epsilon=0$.

\smallskip
\noindent\textit{Optimal bound ($\epsilon=1$, SUIP-DG).}
Let again $e\eqbydef u-\du\in V(h)$ and consider the dual problem: find $\phi\in V$ such that
\begin{equation}
  \label{dual_problem}
  \inerprod{\bkap\grd{\phi}}{\grd{v}}{0}{\domain}
  =
  \inerprod{e}{v}{0}{\domain},
  \qquad\forall v\in V.
\end{equation}
Assume that the dual solution satisfies the elliptic regularity estimate (see, \eg \cite{Brenner2008})
\begin{equation}
  \label{elliptic_reg}
  \seminorm{\phi}{2}{\domain}
  \leq
  \C{reg}\norm{e}{0}{\Th}.
\end{equation}
Let $\proj{\phi}\in\Vh$ denote the elementwise $\LL{2}$-projection of $\phi$ onto $\Vh$. Since the scheme is symmetric for $\epsilon=1$, we have $a_h^{(1)}(e,\phi)=a_h^{(1)}(\phi,e)$. Moreover, by consistency applied to the exact dual solution $\phi$,
\[
a_h^{(1)}(\phi,e)=\inerprod{e}{e}{0}{\Th}=\norm{e}{0}{\Th}^2.
\]
Using the Galerkin orthogonality \eqref{eq:galerkin_orth} with $\dv\eqbydef \proj{\phi}\in\Vh$, we obtain
\[
a_h^{(1)}(e,\proj{\phi})=0,
\]
and hence
\[
\norm{e}{0}{\Th}^2
=
a_h^{(1)}(e,\phi-\proj{\phi}).
\]
Invoking boundedness and the approximation estimate of \cref{lem:approx-star} with $\mu=2$, we infer
\[
\norm{e}{0}{\Th}^2
\leq
\C{bnd}\tnorm{e}\tnorm{\phi-\proj{\phi}}_{*}
\leq
\C{bnd}\C{int}\norm{\bkap}{\infty}{\domain}^{\ohalf}
h\,\tnorm{e}\seminorm{\phi}{2}{\domain}.
\]
Combining this with \eqref{elliptic_reg} and dividing by $\norm{e}{0}{\Th}$ yields \eqref{L2-estimate} with $s_\epsilon=1$ and $\C{}\eqbydef \C{bnd}\C{int}\C{reg}\norm{\bkap}{\infty}{\domain}^{\ohalf}$.
\end{proof}

\section{Numerical experiments}
\label{S:5}

Here we present numerical experiments to validate the theoretical results of \cref{S:4} and assess the practical performance of the UIP-DG method. Two test cases are considered: a heterogeneous anisotropic problem with a smooth manufactured solution (\cref{S:5.1}), and the Kellogg--Rivi\`ere problem with a rough solution (\cref{S:5.2}). We compare the three UIP-DG variants ($\epsilon\in\{0,\pm1\}$) against the Symmetric Weighted IP (SWIP) method of Ern, Stephansen, and Zunino~\cite{ESZIMA08}, which uses diffusivity-based weights $\omega_i^\kappa$ and a single-valued harmonic-mean penalty, and against the single-valued variant IP$_{\mathrm{F}}$-DG of Fabien~\al\cite{FabienNM2019}.

\paragraph{Implementation}
All computations are performed using the finite element library NGSolve \cite{SchoberlNGSolve2014,NGSolveWeb} on sequences of conforming triangular meshes. The stabilization parameter of UIP-DG is set according to \cref{definition:penalty_parameter} with $\alpha_0=8>2(d+1)$, which satisfies the coercivity threshold of \cref{thm:coercivity} for the considered mesh geometries. The polynomial degree ranges from $k=1$ to $k=4$. Errors are measured in the energy norm $\tnorm{\cdot}$ and in the $\LL{2}$-norm $\norm{\cdot}{0}{\Th}$, and estimated convergence rates (ECRs) are computed between consecutive mesh levels. 

\subsection{Test 1 -- Heterogeneous anisotropic diffusion with smooth solution}
\label{S:5.1}

We verify the theoretical convergence rates established in \cref{thm:nrj-estimate,thm:L2-estimates} for all three variants $\epsilon\in\{0,\pm1\}$ of the UIP-DG method, under the high-regularity assumption $u\in\hilbert{k+1}{\domain}{}$. We consider the homogeneous Dirichlet problem on the unit square $\domain\eqbydef(0,1)^2$, with manufactured exact solution
\begin{equation}
\label{exact_sol_test1}
    u(x,y)=\sin(\pi x)\sin(\pi y).
\end{equation}
The domain is partitioned into four non-overlapping subdomains $\domain_i$, $i=1,\ldots,4$, as depicted in \cref{Test-fig-descrip} with anisotropic diffusion tensor $\bkap|_{\domain_i}=\bkap_i$, where
\begin{equation*}
    \bkap_1=\bkap_3=\begin{bmatrix}
    \lambda & 0 \\
    0 & 1
    \end{bmatrix}\quad\textrm{and}\quad\bkap_2=\bkap_4=
     \begin{bmatrix}
    1 & 0 \\
    0 & \lambda^{-1}
    \end{bmatrix}
\end{equation*}
where $\lambda>0$ simultaneously controls the local anisotropy within each subdomain and the global heterogeneity of the medium, with $f=-\dvg(\bkap\grd{u})$ computed accordingly. We consider a sequence of non-uniform triangular meshes conforming to the subdomain partition, so that the discontinuities of $\bkap$ align with element
interfaces.
\begin{figure}[!ht]
    \centering
    \begin{subfigure}
    \centering
    \begin{tikzpicture}[scale=0.5]
	    \draw (0,0) -- (10,0);
    	\draw (10,0) -- (10,10);
    	\draw (10,10) -- (0,10);	
    	\draw (0,10) -- (0,0);
    	
    	\draw[dashed] (0,5) -- (10,5);
    	\draw[dashed] (5,0) -- (5,10);
    	
    	\draw (0,0) node[above right] {$\Omega_1$};
    	\draw (10,0) node[above left] {$\Omega_2$};
    	\draw (10,10) node[below left] {$\Omega_3$};
    	\draw (0,10) node[below right] {$\Omega_4$};
    	
    	\draw[<-> , > = stealth] (1,2.5) -- (4,2.5);
    	\draw[<-> , > = stealth] (2.5,2) -- (2.5,3);
    	\draw (2.5,2.5) node[below right] {$\bkap_{1}$};
    	
    	\draw[<-> , > = stealth] (7,2.5) -- (8,2.5);
    	\draw[<-> , > = stealth] (7.5,1) -- (7.5,4);
    	\draw (7.5,2.5) node[below right] {$\bkap_{2}$};
    	
    	\draw[<-> , > = stealth] (6,7.5) -- (9,7.5);
    	\draw[<-> , > = stealth] (7.5,7) -- (7.5,8);
    	\draw (7.5,7.5) node[below right] {$\bkap_{3}$};
    	
    	\draw[<-> , > = stealth] (2,7.5) -- (3,7.5);
    	\draw[<-> , > = stealth] (2.5,6) -- (2.5,9);
    	\draw (2.5,7.5) node[below right] {$\bkap_{4}$};
    	
        \draw (5,0) node[below] {(a)};
	\end{tikzpicture}\hspace{0.5cm}
    \end{subfigure}
    \begin{subfigure}
    \centering
    \begin{tikzpicture}[scale=0.5]
	    \draw (0,0) -- (10,0);
    	\draw (10,0) -- (10,10);
    	\draw (10,10) -- (0,10);	
    	\draw (0,10) -- (0,0);
    	
    	\draw[dashed] (0,5) -- (10,5);
    	\draw[dashed] (5,0) -- (5,10);
    	
    	\draw (0,0) node[above right] {$\Omega_1$};
    	\draw (10,0) node[above left] {$\Omega_2$};
    	\draw (10,10) node[below left] {$\Omega_3$};
    	\draw (0,10) node[below right] {$\Omega_4$};
    	
    	\draw[<-> , > = stealth] (1,2.5) -- (4,2.5);
        \draw[<-> , > = stealth] (2.5,1) -- (2.5,4);
    	\draw (2.5,2.5) node[below right] {$\bkap_{1}$};
    	
    	\draw[<-> , > = stealth] (7,2.5) -- (8,2.5);
    	\draw[<-> , > = stealth] (7.5,2) -- (7.5,3);
    	\draw (7.5,2.5) node[below right] {$\bkap_{2}$};
    	
    	\draw[<-> , > = stealth] (6,7.5) -- (9,7.5);
    	\draw[<-> , > = stealth] (7.5,6) -- (7.5,9);
    	\draw (7.5,7.5) node[below right] {$\bkap_{3}$};
    	
    	\draw[<-> , > = stealth] (2,7.5) -- (3,7.5);
        \draw[<-> , > = stealth] (2.5,7) -- (2.5,8);
    	\draw (2.5,7.5) node[below right] {$\bkap_{4}$};
    	
        \draw (5,0) node[below] {(b)};
	\end{tikzpicture}\hspace{0.5cm}
    \end{subfigure}
    \caption{Two heterogeneous anisotropic benchmark configurations on the four-subdomain partition $\Omega=\bigcup_{i=1}^4 \Omega_i$. In each subdomain $\Omega_i$, the tensor $\bkap_i$ is represented through its principal directions and relative anisotropy. (a) Test~1: combined anisotropy and heterogeneity with a smooth solution, used to assess optimal convergence. (b) Test~2: heterogeneous configuration with a singular solution exhibiting reduced regularity across interfaces.}
    \label{Test-fig-descrip}
\end{figure}
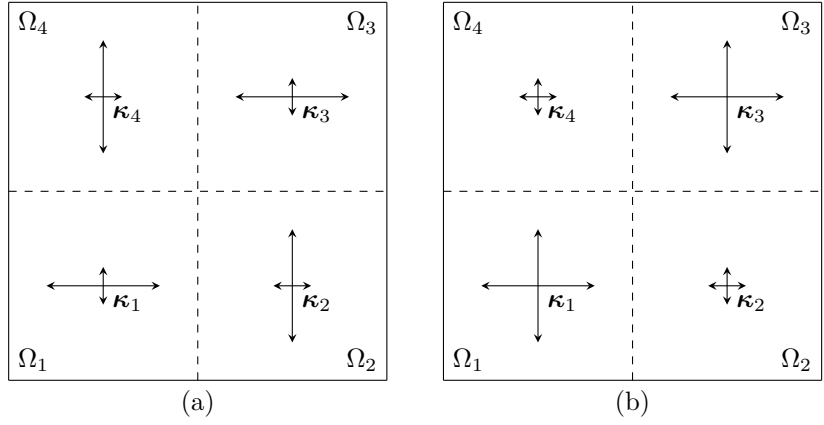

\paragraph{Convergence rates}
The convergence histories for $\lambda=10^4$ are reported in \cref{convergence1}. For all polynomial degrees $k\in\{1,\ldots,4\}$, UIP-DG achieves the optimal $\tnorm{\cdot}$-norm rate $\bigO{h^k}$ predicted by \cref{thm:nrj-estimate}, uniformly across variants. All three variants exhibit identical asymptotic convergence rates in the energy norm. The parameter $\epsilon$ mainly affects the symmetry of the formulation and the associated constants, but does not alter the convergence orders. For the symmetric variant (SUIP-DG, $\epsilon=1$), the $\LL{2}$-error converges at the optimal rate $\bigO{h^{k+1}}$ in agreement with \cref{thm:L2-estimates}, while IUIP- and NUIP-DG exhibit the expected rate $\bigO{h^k}$. These results hold uniformly for $\lambda\in\{1,10^4\}$, supporting the contrast-robustness of the stability constants established in \cref{rem:contrast_indep}.

\paragraph{Robustness under high contrast}
To assess robustness with respect to diffusion contrast, we fix $h=1/32$, $k=1$, and $\lambda=10^4$, and compare in \cref{Test_comparative} UIP-DG with IP$_{\mathrm{F}}$-DG (for $\epsilon\in\{0,\pm1\}$). The UIP-DG solutions remain stable and free of visible oscillations across material interfaces. In contrast, the IP$_{\mathrm{F}}$-DG solution exhibits oscillations localized near the interface between subdomains with strongly differing diffusion coefficients, whose magnitude increases with $\lambda$. This behavior is consistent with the structure of the interface terms. In IP$_{\mathrm{F}}$-DG, arithmetic averaging and single-valued penalization do not reflect the scaling induced by heterogeneous diffusion, which leads to an imbalance in the discrete flux transmission. In UIP-DG, the transmissibility-based weights $\omega_i^\tau$ and the additional control of flux jumps enforce a consistent scaling of interface contributions, yielding a more robust behavior in high-contrast regimes. 

\begin{figure}[!ht]
\centering
\subfigure{\begin{tikzpicture}[scale=0.47]
\begin{loglogaxis}[title=\normalsize{\textbf{NUIP}}, ylabel=\normalsize{$\tnorm{u-\du}$},xlabel=\normalsize{$h$},legend cell align=left, legend pos = south east, font=\small,line width=1.25pt]

\addplot[color=MidnightBlue,mark=oplus*] coordinates {(1/8,1.9e+01) (1/16,9.5e+00) (1/32,4.8e+00) (1/64,2.4e+00) (1/128,1.2e+00) }; 
		
\addplot[color=ForestGreen,mark=triangle*] coordinates { (1/8,1.9e+00) (1/16,4.9e-01) (1/32,1.2e-01) (1/64,3.2e-02) (1/128,7.9e-03) }; 

\addplot[color=carmine,mark=square*] coordinates {(1/8,1.1e-01) (1/16,1.3e-02) (1/32,1.7e-03) (1/64,2.1e-04) (1/128,2.6e-05) }; 

\addplot[color=burntorange,mark=diamond*] coordinates { (1/8,5.4e-03) (1/16,3.4e-04) (1/32,2.1e-05) (1/64,1.3e-06) (1/128,8.3e-08) }; 


\draw[color=black,dashed] (axis cs:1/64,2.4e+00) |- (axis cs:1/128,1.2e+00) node[near start,right]{$1.0$};

\draw[color=black,dashed] (axis cs:1/64,3.2e-02) |- (axis cs:1/128,7.9e-03) node[near start,right]{$2.0$};

\draw[color=black,dashed] (axis cs:1/64,2.1e-04) |- (axis cs:1/128,2.6e-05) node[near start,right]{$3.0$};

\draw[color=black,dashed] (axis cs:1/64,1.3e-06) |- (axis cs:1/128,8.3e-08) node[near start,right]{$4.0$};

\legend{$k=1$, $k=2$, $k=3$, $k=4$}

\end{loglogaxis}
\end{tikzpicture} }
\subfigure{\begin{tikzpicture}[scale=0.47]
\begin{loglogaxis}[title=\normalsize{\textbf{IUIP}},xlabel=\normalsize{$h$},legend cell align=left, legend pos = south east, font=\small,line width=1.25pt]

\addplot[color=MidnightBlue,mark=oplus*,mark options=solid] coordinates {(1/8,1.9e+01) (1/16,9.5e+00) (1/32,4.8e+00) (1/64,2.4e+00) (1/128,1.2e+00) }; 
		
\addplot[color=ForestGreen,mark=triangle*,mark options=solid] coordinates {(1/8,1.9e+00) (1/16,4.8e-01) (1/32,1.2e-01) (1/64,3.1e-02) (1/128,7.8e-03) }; 

\addplot[color=carmine,mark=square*,mark options=solid] coordinates {(1/8,1.1e-01) (1/16,1.3e-02) (1/32,1.6e-03) (1/64,2.0e-04) (1/128,2.5e-05) }; 

\addplot[color=burntorange,mark=diamond*,mark options=solid] coordinates { (1/8,5.3e-03) (1/16,3.4e-04) (1/32,2.1e-05) (1/64,1.3e-06) (1/128,8.3e-08) }; 


\draw[color=black,dashed] (axis cs:1/64,2.4e+00) |- (axis cs:1/128,1.2e+00) node[near start,right]{$1.0$};

\draw[color=black,dashed] (axis cs:1/64,3.1e-02) |- (axis cs:1/128,7.8e-03) node[near start,right]{$2.0$};

\draw[color=black,dashed] (axis cs:1/64,2.0e-04) |- (axis cs:1/128,2.5e-05) node[near start,right]{$3.0$};

\draw[color=black,dashed] (axis cs:1/64,1.3e-06) |- (axis cs:1/128,8.3e-08) node[near start,right]{$4.0$};

\legend{$k=1$, $k=2$, $k=3$, $k=4$}

\end{loglogaxis}
\end{tikzpicture} }
\subfigure{\begin{tikzpicture}[scale=0.47]
\begin{loglogaxis}[title=\normalsize{\textbf{SUIP}},xlabel=\normalsize{$h$},legend cell align=left, legend pos = south east, font=\small,line width=1.25pt]

\addplot[color=MidnightBlue,mark=oplus*,mark options=solid] coordinates { (1/8,1.9e+01) (1/16,9.8e+00) (1/32,4.9e+00) (1/64,2.5e+00) (1/128,1.2e+00) }; 
		
\addplot[color=ForestGreen,mark=triangle*,mark options=solid] coordinates {(1/8,2.1e+00) (1/16,5.3e-01) (1/32,1.3e-01) (1/64,3.4e-02) (1/128,8.5e-03) }; 

\addplot[color=carmine,mark=square*,mark options=solid] coordinates {(1/8,1.2e-01) (1/16,1.5e-02) (1/32,1.9e-03) (1/64,2.4e-04) (1/128,2.9e-05) }; 

\addplot[color=burntorange,mark=diamond*,mark options=solid] coordinates { (1/8,6.5e-03) (1/16,4.1e-04) (1/32,2.6e-05) (1/64,1.6e-06) (1/128,1.0e-07) }; 


\draw[color=black,dashed] (axis cs:1/64,2.5e+00) |- (axis cs:1/128,1.2e+00) node[near start,right]{$1.0$};

\draw[color=black,dashed] (axis cs:1/64,3.4e-02) |- (axis cs:1/128,8.5e-03) node[near start,right]{$2.0$};

\draw[color=black,dashed] (axis cs:1/64,2.4e-04) |- (axis cs:1/128,2.9e-05) node[near start,right]{$3.0$};

\draw[color=black,dashed] (axis cs:1/64,1.6e-06) |- (axis cs:1/128,1.0e-07) node[near start,right]{$4.0$};

\legend{$k=1$, $k=2$, $k=3$, $k=4$}

\end{loglogaxis}
\end{tikzpicture} }\\ 
\subfigure{\begin{tikzpicture}[scale=0.47]
\begin{loglogaxis}[ylabel=\normalsize{$\norm{u-\du}{0}{\Th}$},xlabel=\normalsize{$h$},legend cell align=left, legend pos = south east, font=\small,line width=1.25pt]

\addplot[color=MidnightBlue,mark=oplus*,mark options=solid,dashed] coordinates {(1/8,1.5e-02) (1/16,4.7e-03) (1/32,1.7e-03) (1/64,4.8e-04) (1/128,1.0e-04) }; 
		
\addplot[color=ForestGreen,mark=triangle*,mark options=solid,dashed] coordinates {(1/8,2.9e-03) (1/16,7.2e-04) (1/32,1.8e-04) (1/64,4.5e-05) (1/128,1.1e-05) }; 

\addplot[color=carmine,mark=square*,mark options=solid,dashed] coordinates { (1/8,7.7e-05) (1/16,7.5e-06) (1/32,8.2e-07) (1/64,7.9e-08) (1/128,5.3e-09) }; 

\addplot[color=burntorange,mark=diamond*,mark options=solid,dashed] coordinates {(1/8,4.1e-06) (1/16,2.4e-07) (1/32,1.5e-08) (1/64,8.9e-10) (1/128,5.2e-11) };  


\draw[color=black,dashed] (axis cs:1/64,4.5e-05) |- (axis cs:1/128,1.1e-05) node[near start,right]{$2.0$};

\draw[color=black,dashed] (axis cs:1/64,8.9e-10) |- (axis cs:1/128,5.2e-11) node[near start,right]{$4.1$};

\legend{$k=1$, $k=2$, $k=3$, $k=4$}

\end{loglogaxis}
\end{tikzpicture} }
\subfigure{\begin{tikzpicture}[scale=0.47]
\begin{loglogaxis}[xlabel=\normalsize{$h$},legend cell align=left, legend pos = south east, font=\small,line width=1.25pt]

\addplot[color=MidnightBlue,mark=oplus*,mark options=solid,dashed] coordinates { (1/8,1.3e-02) (1/16,3.7e-03) (1/32,1.1e-03) (1/64,3.3e-04) (1/128,7.6e-05) }; 
		
\addplot[color=ForestGreen,mark=triangle*,mark options=solid,dashed] coordinates { (1/8,2.0e-03) (1/16,4.7e-04) (1/32,1.2e-04) (1/64,2.9e-05) (1/128,7.3e-06) }; 

\addplot[color=carmine,mark=square*,mark options=solid,dashed] coordinates {(1/8,5.9e-05) (1/16, 5.6e-06) (1/32,6.1e-07) (1/64,5.9e-08) (1/128,4.0e-09) }; 

\addplot[color=burntorange,mark=diamond*,mark options=solid,dashed] coordinates {(1/8,3.1e-06) (1/16,1.8e-07) (1/32,1.1e-08) (1/64,6.3e-10) (1/128,3.6e-11) }; 


\draw[color=black,dashed] (axis cs:1/64,2.9e-05) |- (axis cs:1/128,7.3e-06) node[near start,right]{$2.0$};

\draw[color=black,dashed] (axis cs:1/64,6.3e-10) |- (axis cs:1/128,3.6e-11) node[near start,right]{$4.1$};

\legend{$k=1$, $k=2$, $k=3$, $k=4$}

\end{loglogaxis}
\end{tikzpicture} }
\subfigure{\begin{tikzpicture}[scale=0.47]
\begin{loglogaxis}[xlabel=\normalsize{$h$},legend cell align=left, legend pos = south east, font=\small,line width=1.25pt]

\addplot[color=MidnightBlue,mark=oplus*,mark options=solid,dashed] coordinates { (1/8,1.3e-02) (1/16,3.8e-03) (1/32,1.1e-03) (1/64,2.9e-04) (1/128,6.4e-05) }; 
		
\addplot[color=ForestGreen,mark=triangle*,mark options=solid,dashed] coordinates { (1/8,8.0e-04) (1/16,1.1e-04) (1/32,1.5e-05) (1/64,2.2e-06) (1/128,2.8e-07) }; 

\addplot[color=carmine,mark=square*,mark options=solid,dashed] coordinates { (1/8,4.5e-05) (1/16,2.9e-06) (1/32,1.8e-07) (1/64,1.1e-08) (1/128,7.5e-10) }; 

\addplot[color=burntorange,mark=diamond*,mark options=solid,dashed] coordinates { (1/8,2.0e-06) (1/16,6.7e-08) (1/32,2.3e-09) (1/64,7.2e-11) (1/128,2.8e-12) }; 


\draw[color=black,dashed] (axis cs:1/64,2.9e-04) |- (axis cs:1/128,6.4e-05) node[near start,right]{$2.2$};

\draw[color=black,dashed] (axis cs:1/64,2.2e-06) |- (axis cs:1/128,2.8e-07) node[near start,right]{$2.9$};

\draw[color=black,dashed] (axis cs:1/64,1.1e-08) |- (axis cs:1/128,7.5e-10) node[near start,right]{$3.9$};

\draw[color=black,dashed] (axis cs:1/64,7.2e-11) |- (axis cs:1/128,2.8e-12) node[near start,right]{$4.7$};

\legend{$k=1$, $k=2$, $k=3$, $k=4$}

\end{loglogaxis}
\end{tikzpicture}}\vspace{-0.75cm}
\caption{Test 1 -- History of convergence of the UIP-DG method in the $\tnorm{\cdot}$-norm (continuous line) and  $\LL{2}$-norm (dashed line) versus the mesh-size $h$ for polynomial degrees $k\in\{1,\ldots,4\}$, respectively.}
\label{convergence1}
\end{figure}
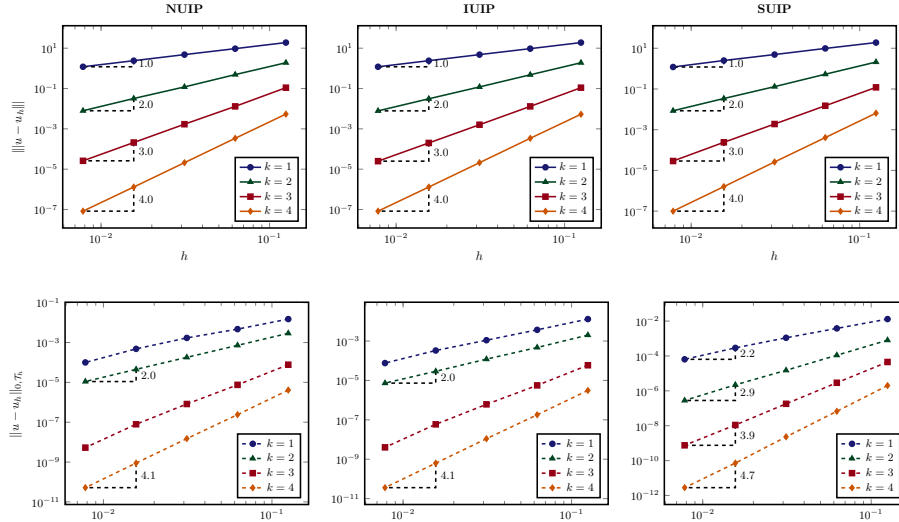

\begin{figure}[!ht]
    \hspace{1.6cm}\tiny{\textbf{NUIP}}\hspace{3.4cm}\tiny{\textbf{IUIP}}\hspace{3.5cm}\tiny{\textbf{SUIP}}\\

    \begin{subfigure}
        \centering
        \safeincludegraphics[scale=0.131,trim = 490 20 490 110, clip=true]{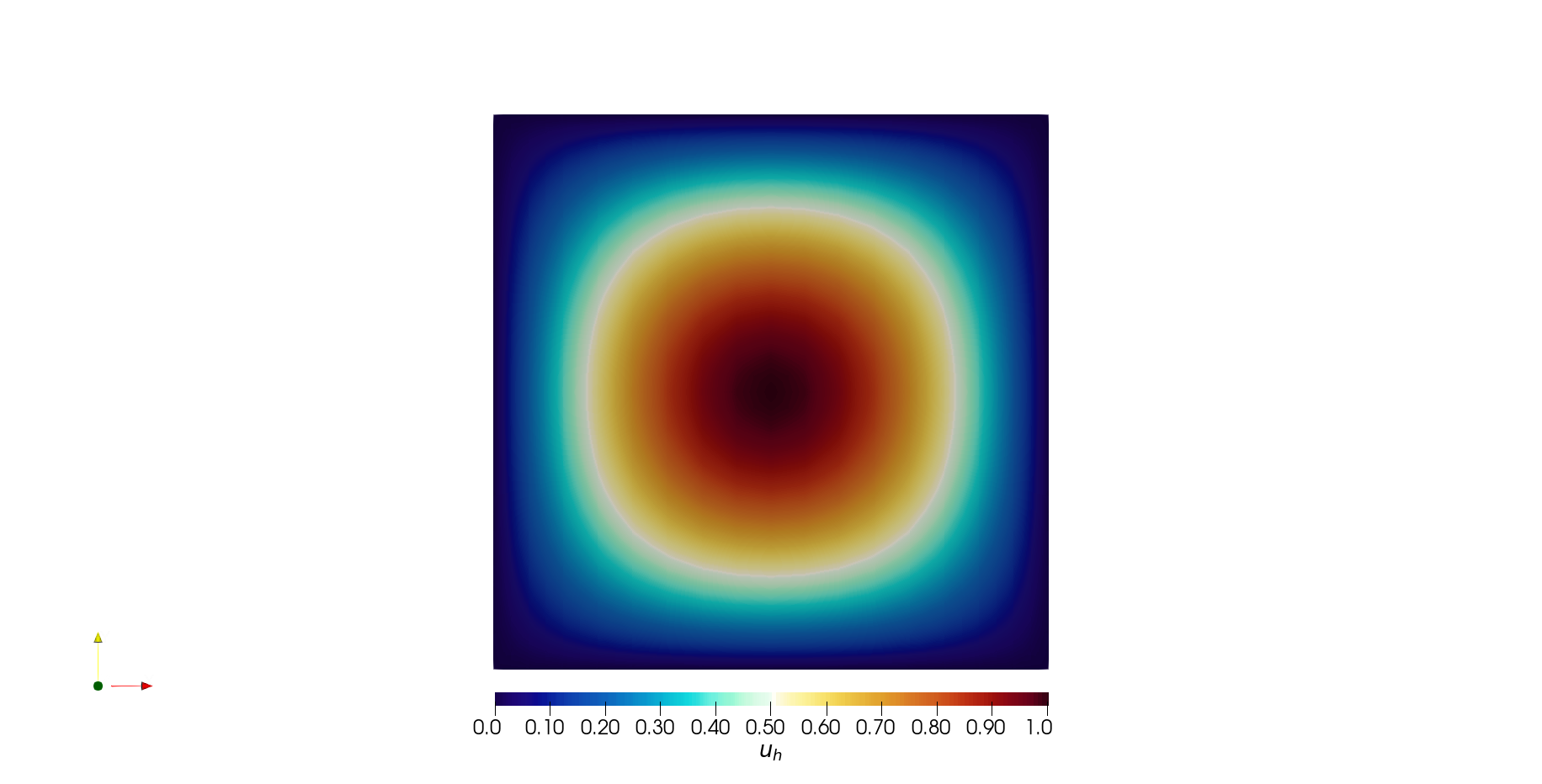}
    \end{subfigure}
    \begin{subfigure}
        \centering
        \safeincludegraphics[scale=0.131,trim = 490 20 490 110, clip=true]{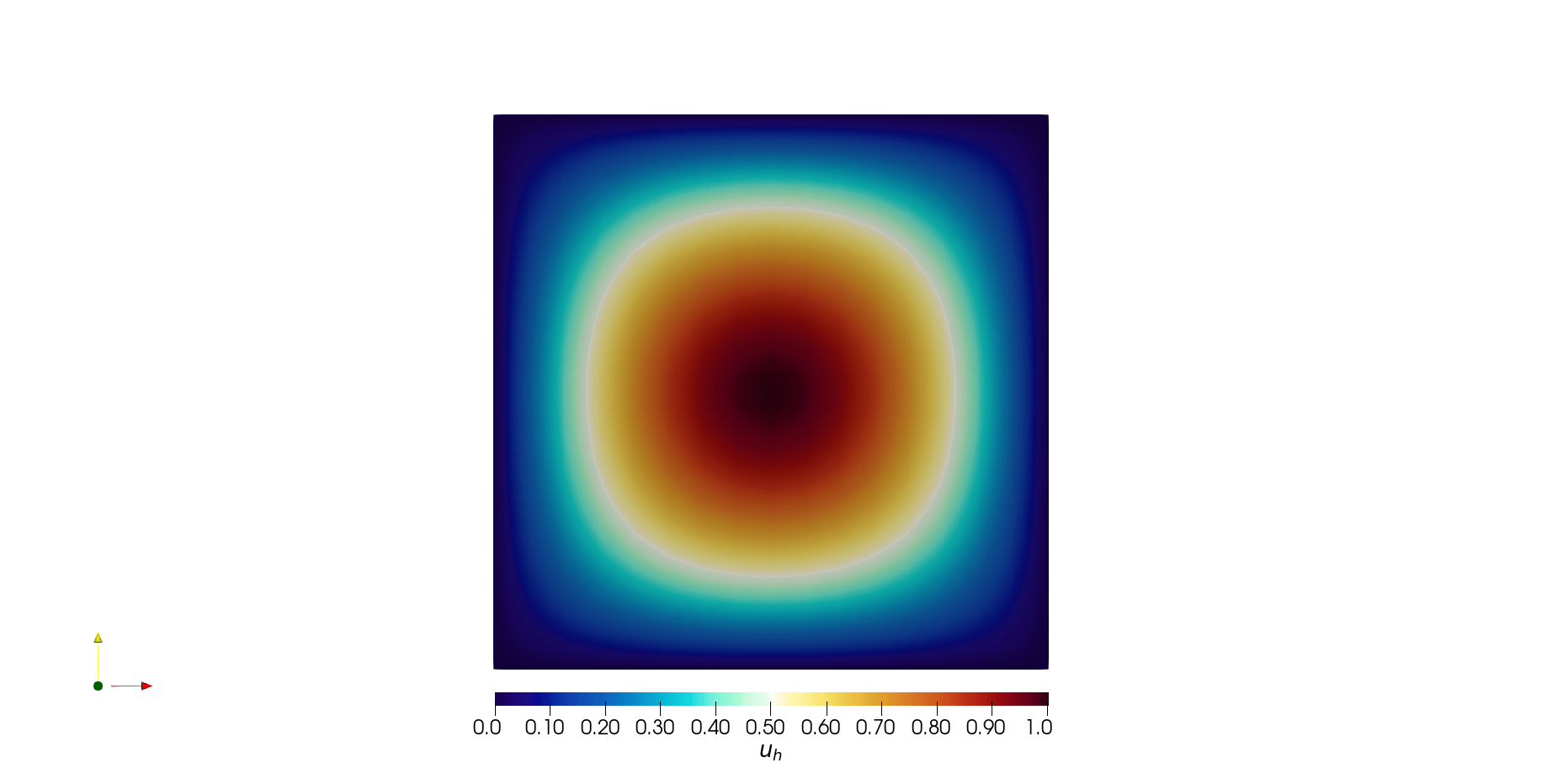}
    \end{subfigure}
    \begin{subfigure}
        \centering
        \safeincludegraphics[scale=0.131,trim = 490 20 490 110, clip=true]{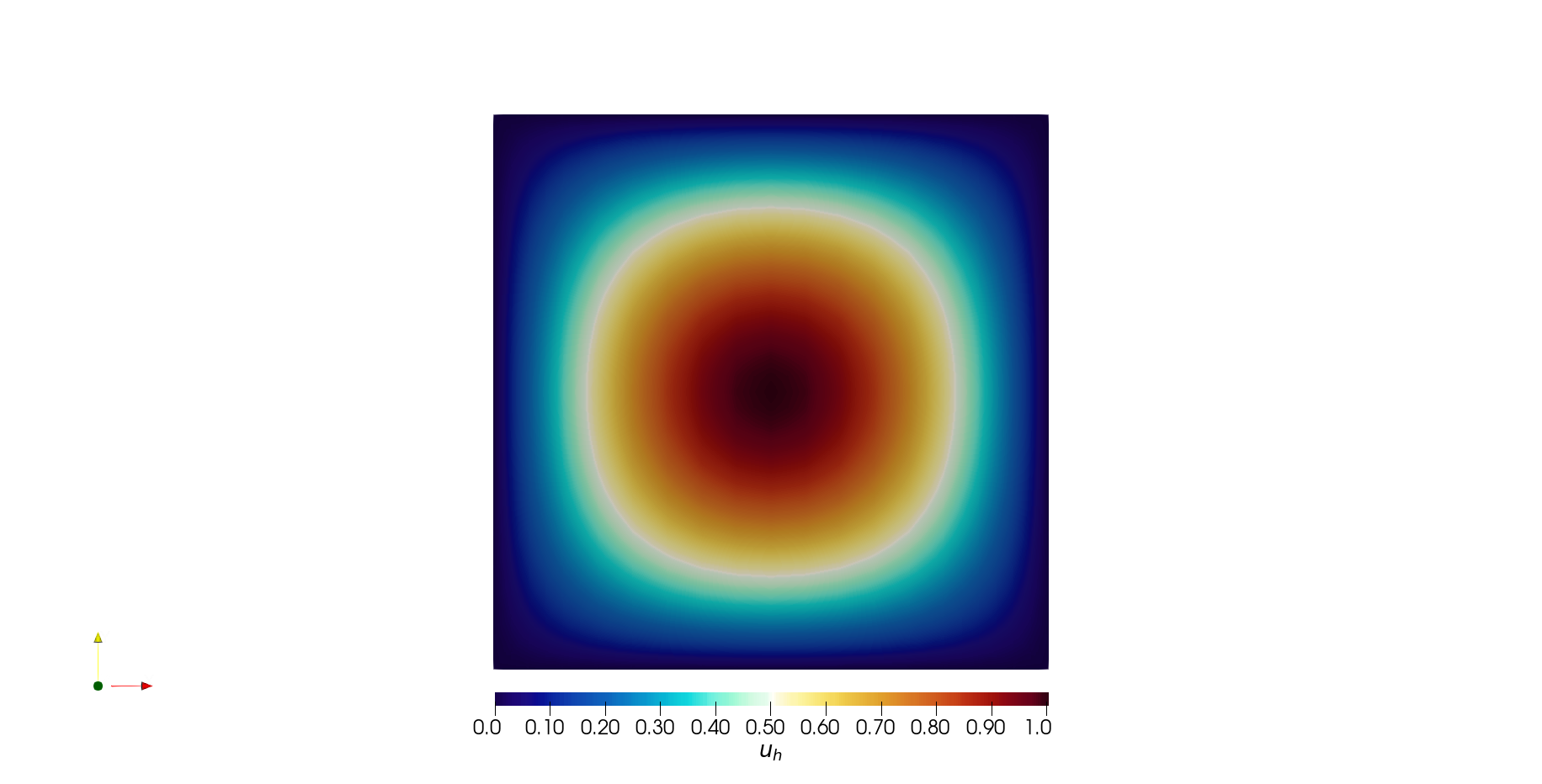}
    \end{subfigure}\\
    
    \hspace{1.6cm}\tiny{\textbf{NIP$_{\textrm{F}}$}}\hspace{3.4cm}\tiny{\textbf{IIP$_{\textrm{F}}$}}\hspace{3.5cm}\tiny{\textbf{SIP$_{\textrm{F}}$}}\\

    \begin{subfigure}
        \centering
        \safeincludegraphics[scale=0.131,trim = 490 20 490 110, clip=true]{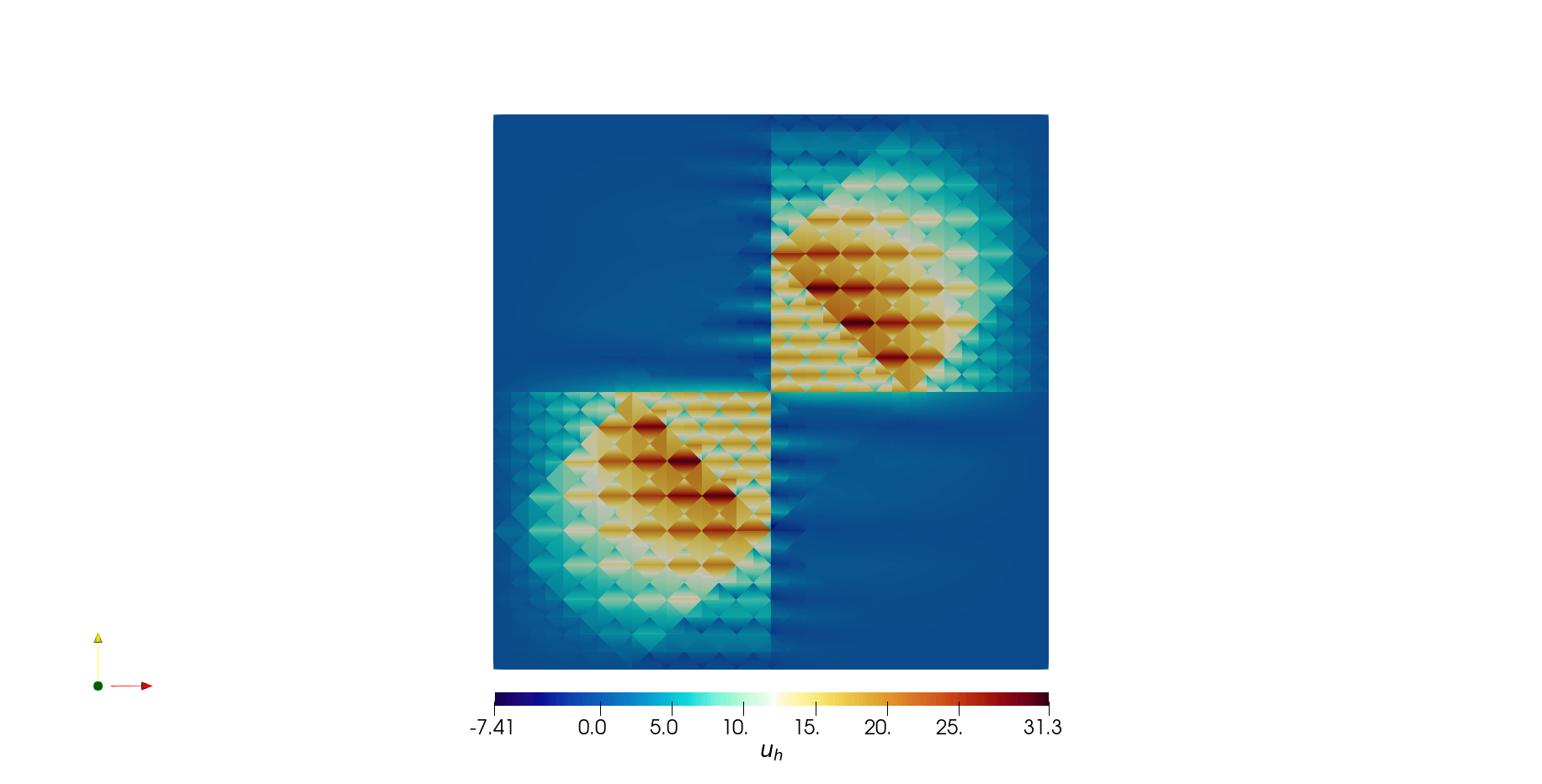}
    \end{subfigure}
    \begin{subfigure}
        \centering
        \safeincludegraphics[scale=0.131,trim = 490 20 490 110, clip=true]{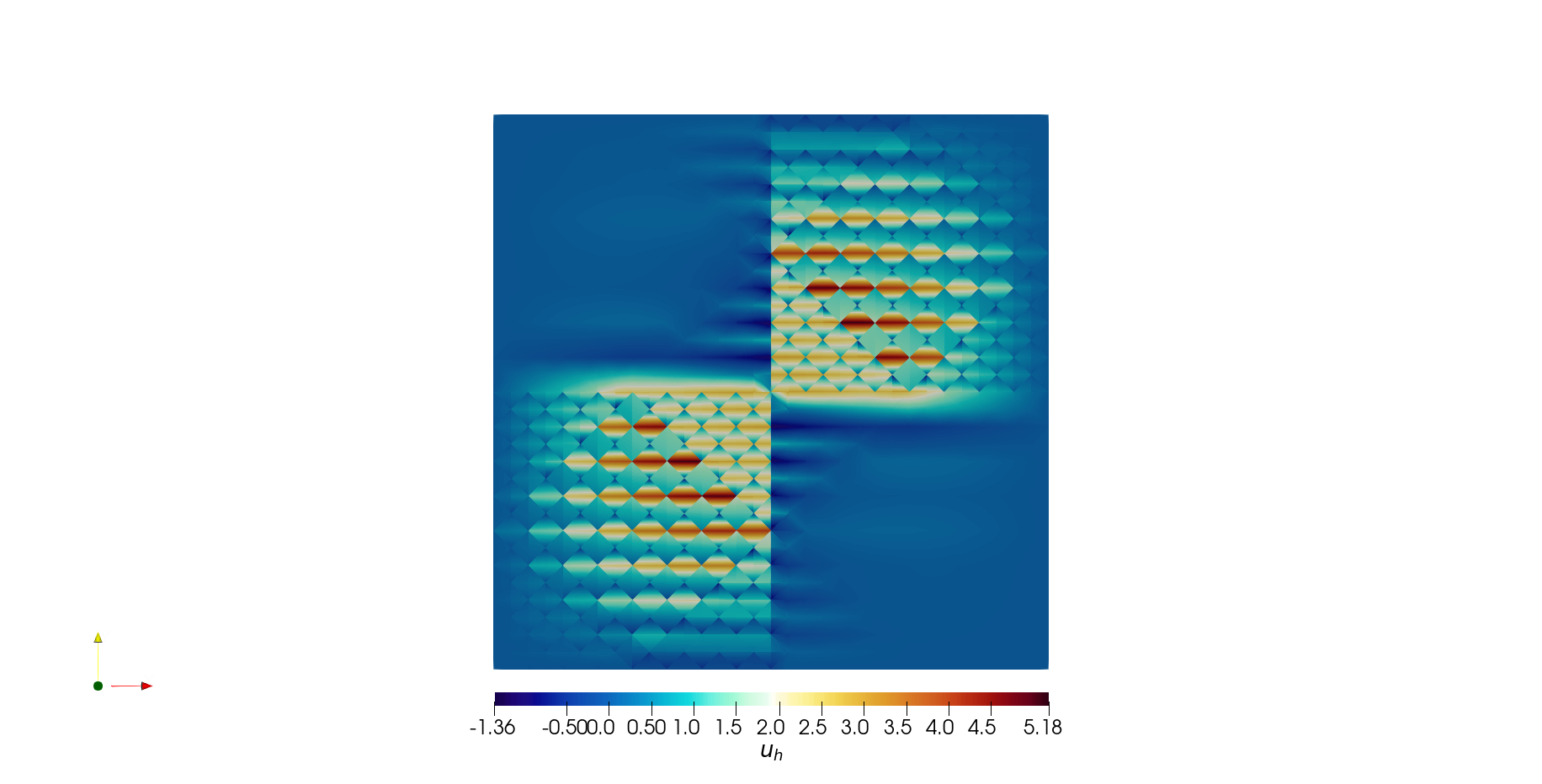}
    \end{subfigure}
    \begin{subfigure}
        \centering
        \safeincludegraphics[scale=0.131,trim = 490 20 490 110, clip=true]{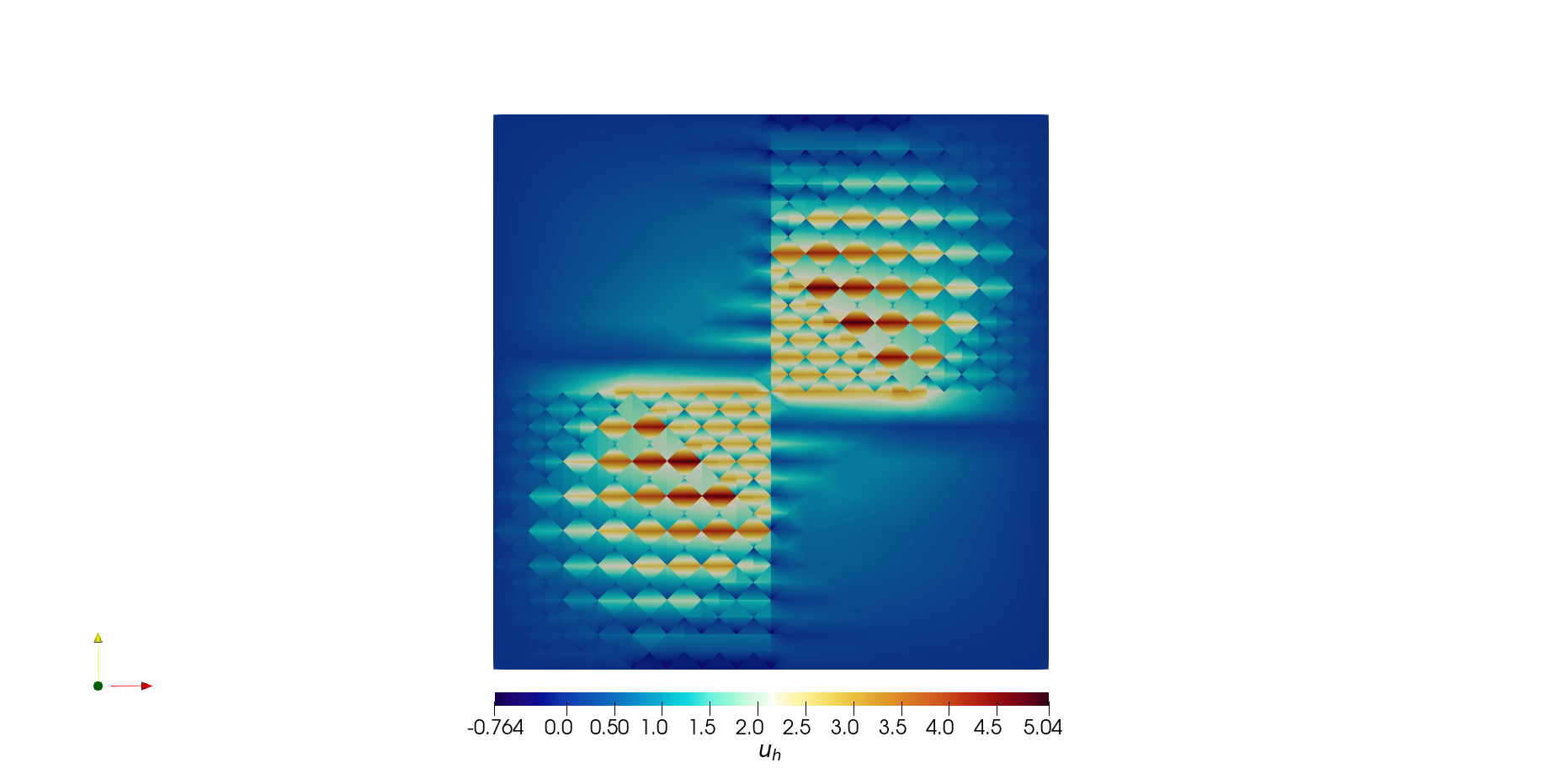}
    \end{subfigure}
\caption{Test 1 -- Representation of the discrete piecewise linear solution $\du$ obtained by the UIP-DG method (first line) and the IP$_{\textrm{F}}$-DG version (second line) for all variations of $\epsilon\in\{0,\pm1\}$. 
}
\label{Test_comparative}
\end{figure}

\subsection{Test 2 -- Heterogeneous isotropic diffusion with rough solution}
\label{S:5.2}

We now consider a test case with a rough solution, following Kellogg~\cite{Kellogg1975} and Riviere~\cite{Riviere08bookDG}. The domain $\domain=(-1,1)^2$ is partitioned into four non-overlapping subdomains $\domain_i$, $i=1,\ldots,4$, as depicted in \cref{Test-fig-descrip}(b), with isotropic diffusion tensor $\bkap|_{\domain_i}=\kappa_i\mathbf{I}$, where
\begin{equation*}
  \kappa_1=\kappa_3=5, \quad\textrm{and}\quad \kappa_2=\kappa_4=1.
\end{equation*}
The source term vanishes ($f=0$), and non-homogeneous Dirichlet boundary conditions are imposed on $\Gamma_D=\partial\domain$. The exact solution is given in polar coordinates by
\begin{equation}
\label{exact_sol_kellogg}
  u|_{\domain_i}(r,\theta)
  = r^\alpha\bigl(a_i\sin(\alpha\theta)+b_i\cos(\alpha\theta)\bigr),
\end{equation}
with coefficients listed in \cref{tab:kellogg_coeffs} and $\alpha=0.5354409456$. Since $\alpha<1$, the exact solution belongs to $H^{1+\alpha}(\domain)$ but not to $H^2(\domain)$, with a corner singularity at the origin. Consequently, the convergence rate is limited by the regularity of $u$ rather than the polynomial degree, and all variants of UIP-DG are expected to converge at the reduced rate $\bigO{h^\alpha}$ in the $\tnorm{\cdot}$-norm for all $k\geq 1$. This test therefore assesses the robustness of UIP-DG in the low-regularity regime, complementing the optimal-rate results of \cref{S:5.1}.
\begin{table}[ht]
\centering
\small
\setlength{\tabcolsep}{2pt}
\caption{Test 2 -- Coefficients of the exact solution for the Kellogg test.}
\label{tab:kellogg_coeffs}
\begin{tabular}{ccccc}
\hline
$i$ & 1 & 2 & 3 & 4 \\
\hline
$a_i$ & $0.4472$ & $-0.7454$ & $-0.9441$ & $-2.4017$ \\
$b_i$ & $1.0000$ & $2.3333$ & $0.5556$ & $-0.4815$ \\
\hline
\end{tabular}
\end{table}

\paragraph{Convergence rates}
For this test case, computations were carried out for several polynomial degrees $k=1,\ldots,4$ and for the three UIP-DG variants. For the sake of conciseness, only the results corresponding to $k=2$ are reported in \Cref{tab:KR_compare_k2}, as they provide a representative illustration of the overall convergence behavior. In particular, all polynomial degrees exhibit the same asymptotic rate, which is limited by the singular regularity of the exact solution. Owing to the singular behavior of the exact solution, characterized by the exponent $\alpha=0.53544$, the energy error converges with rate approximately $0.54$, in excellent agreement with the theoretical estimate, \ie $\tnorm{u-\du}\lesssim h^{\alpha}$. This confirms that the asymptotic regime is governed by the interface singularity, rather than by the local approximation order. While increasing the polynomial degree does not improve the asymptotic rate, it significantly reduces the error constants, highlighting the benefit of high-order approximations even in singular regimes. In the $\LL{2}$-norm, IUIP-DG and NUIP-DG exhibit slightly smaller errors and larger observed convergence rates on this benchmark. Since this trend is not directly predicted by the present analysis, it is only reported here as a numerical observation.
\begin{table}[htbp]
\centering
\small
\setlength{\tabcolsep}{2pt}
\caption{Test 2 -- Comparison of the three UIP-DG variants for the Kellogg--Rivi\`ere test
with polynomial degree $k=2$. All variants recover the singular energy rate
$\alpha \approx 0.54$.}
\label{tab:KR_compare_k2}
\begin{tabular}{cccccc}
\toprule
Method & $h$ & $\norm{u-\du}{0}{\Th}$ & ECR & $\tnorm{u-\du}$ & ECR \\
\midrule
 
& $5.1\!\times\!10^{-2}$ & $1.41\!\times\!10^{-3}$ & --   & $2.78\!\times\!10^{-1}$ & --   \\
SUIP-DG & $2.5\!\times\!10^{-2}$ & $6.36\!\times\!10^{-4}$ & 1.15 & $1.92\!\times\!10^{-1}$ & 0.54 \\
& $1.3\!\times\!10^{-2}$ & $2.94\!\times\!10^{-4}$ & 1.12 & $1.32\!\times\!10^{-1}$ & 0.54 \\
& $6.4\!\times\!10^{-3}$ & $1.37\!\times\!10^{-4}$ & 1.10 & $9.14\!\times\!10^{-2}$ & 0.54 \\
\midrule
 
& $5.1\!\times\!10^{-2}$ & $1.56\!\times\!10^{-3}$ & --   & $2.76\!\times\!10^{-1}$ & --   \\
IUIP-DG & $2.5\!\times\!10^{-2}$ & $5.85\!\times\!10^{-4}$ & 1.41 & $1.91\!\times\!10^{-1}$ & 0.54 \\
& $1.3\!\times\!10^{-2}$ & $2.31\!\times\!10^{-4}$ & 1.34 & $1.31\!\times\!10^{-1}$ & 0.54 \\
& $6.4\!\times\!10^{-3}$ & $9.59\!\times\!10^{-5}$ & 1.27 & $9.07\!\times\!10^{-2}$ & 0.54 \\
\midrule

& $5.1\!\times\!10^{-2}$ & $1.36\!\times\!10^{-3}$ & --   & $2.81\!\times\!10^{-1}$ & --   \\
NUIP-DG & $2.5\!\times\!10^{-2}$ & $4.73\!\times\!10^{-4}$ & 1.52 & $1.94\!\times\!10^{-1}$ & 0.54 \\
& $1.3\!\times\!10^{-2}$ & $1.68\!\times\!10^{-4}$ & 1.49 & $1.34\!\times\!10^{-1}$ & 0.54 \\
& $6.4\!\times\!10^{-3}$ & $6.14\!\times\!10^{-5}$ & 1.46 & $9.23\!\times\!10^{-2}$ & 0.54 \\
\bottomrule
\end{tabular}
\end{table}

\paragraph{Comparison with SWIP}

We compare the symmetric UIP-DG method (SUIP-DG) with the classical symmetric weighted interior penalty (SWIP) method \cite{ESZIMA08} on the Kellogg--Rivi\`ere benchmark. All computations are performed using piecewise affine polynomials ($k=1$). In the energy norm, both methods exhibit identical convergence rates, with an observed rate of approximately $0.54$, in agreement with the theoretical regularity index $\alpha=0.53544$. The error levels are also very close, with SUIP-DG yielding slightly smaller values, typically by about $3\%$ on all refinement levels. This confirms that the additional flux-jump stabilization present in SUIP-DG does not alter the asymptotic behavior of the method.  In contrast, a more pronounced difference is observed in the $\LL{2}$-norm. While both methods display similar convergence trends, SUIP-DG systematically produces smaller errors than SWIP, with a reduction that becomes more significant under mesh refinement. On the finest meshes, the $\LL{2}$-error is reduced by approximately $30\%$. Since this improvement concerns the error constant and is not covered by the present analysis, we report it here as a numerical observation.
\begin{table}[htbp]
\centering
\small
\setlength{\tabcolsep}{2pt}
\caption{Test 2 -- Comparison of SWIP and SUIP-DG for the Kellogg--Rivi\`ere benchmark using piecewise affine polynomials ($k=1$).}
\label{tab:KR_SWIP_SUIP}
\begin{tabular}{cccccccccc}
\toprule
& \multicolumn{4}{c}{SWIP-DG} & \multicolumn{4}{c}{SUIP-DG} \\
\cmidrule(lr){2-5} \cmidrule(lr){6-9}
$h$ 
& $\norm{u-\du}{0}{\Th}$ & ECR 
& $\tnorm{u-\du}$ & ECR
& $\norm{u-\du}{0}{\Th}$ & ECR 
& $\tnorm{u-\du}$ & ECR \\
\midrule
$5.1\!\times\!10^{-2}$ 
& $4.62\!\times\!10^{-3}$ & --- 
& $5.61\!\times\!10^{-1}$ & ---
& $3.83\!\times\!10^{-3}$ & --- 
& $5.44\!\times\!10^{-1}$ & --- \\
$2.5\!\times\!10^{-2}$ 
& $1.67\!\times\!10^{-3}$ & 1.47 
& $3.86\!\times\!10^{-1}$ & 0.54
& $1.30\!\times\!10^{-3}$ & 1.56 
& $3.75\!\times\!10^{-1}$ & 0.54 \\
$1.3\!\times\!10^{-2}$ 
& $6.41\!\times\!10^{-4}$ & 1.38 
& $2.66\!\times\!10^{-1}$ & 0.54
& $4.67\!\times\!10^{-4}$ & 1.48 
& $2.59\!\times\!10^{-1}$ & 0.54 \\
$6.4\!\times\!10^{-3}$ 
& $2.61\!\times\!10^{-4}$ & 1.29 
& $1.84\!\times\!10^{-1}$ & 0.54
& $1.78\!\times\!10^{-4}$ & 1.39 
& $1.79\!\times\!10^{-1}$ & 0.54 \\
\bottomrule
\end{tabular}
\end{table}

\begin{figure}[!ht]
  \centering
  \begin{subfigure}
    \centering
    \safeincludegraphics[scale=0.135,trim = 500 20 490 110, clip=true]{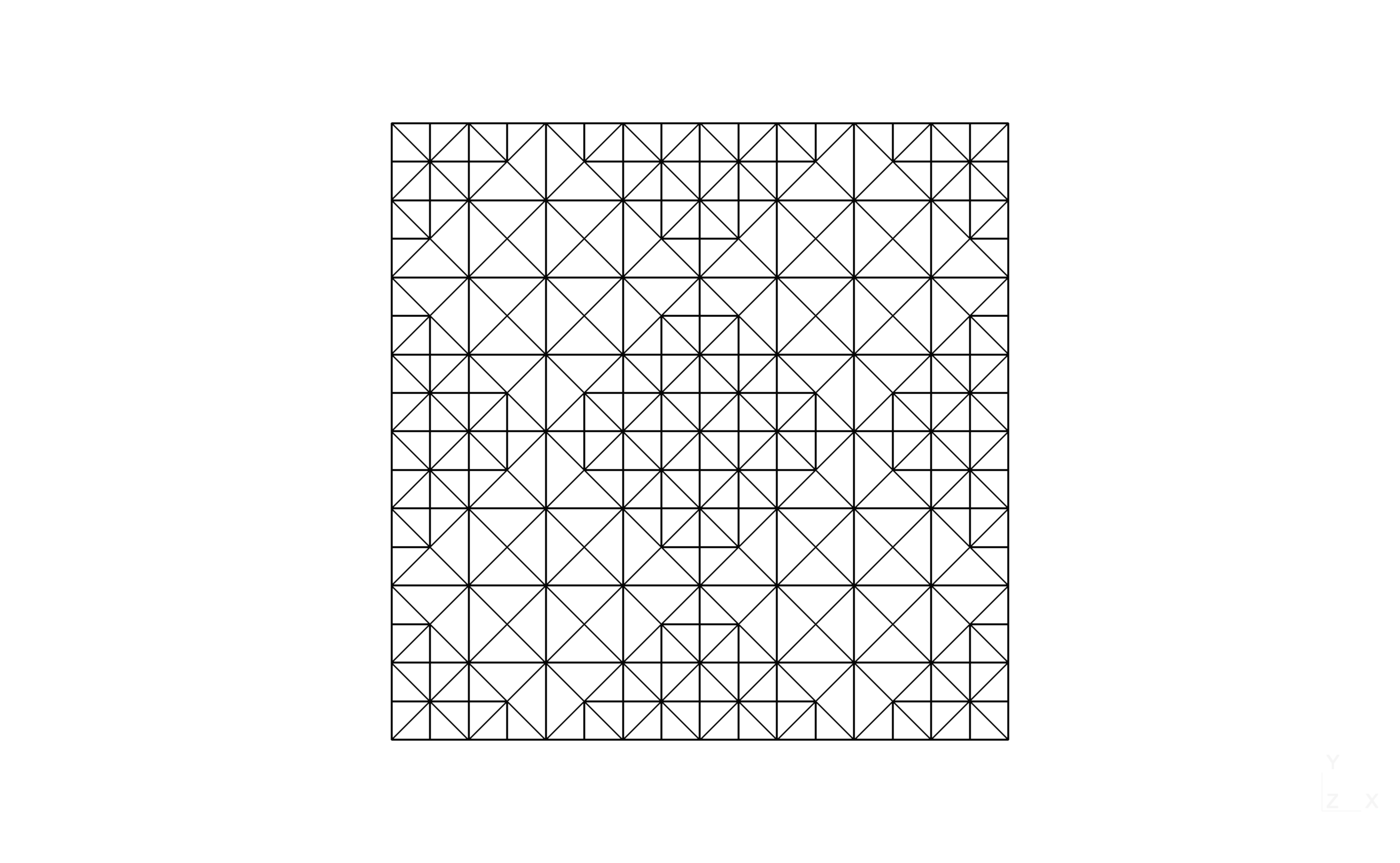}
  \end{subfigure}
  \begin{subfigure}
    \centering
    \safeincludegraphics[scale=0.235,trim = 490 20 323 105, clip=true]{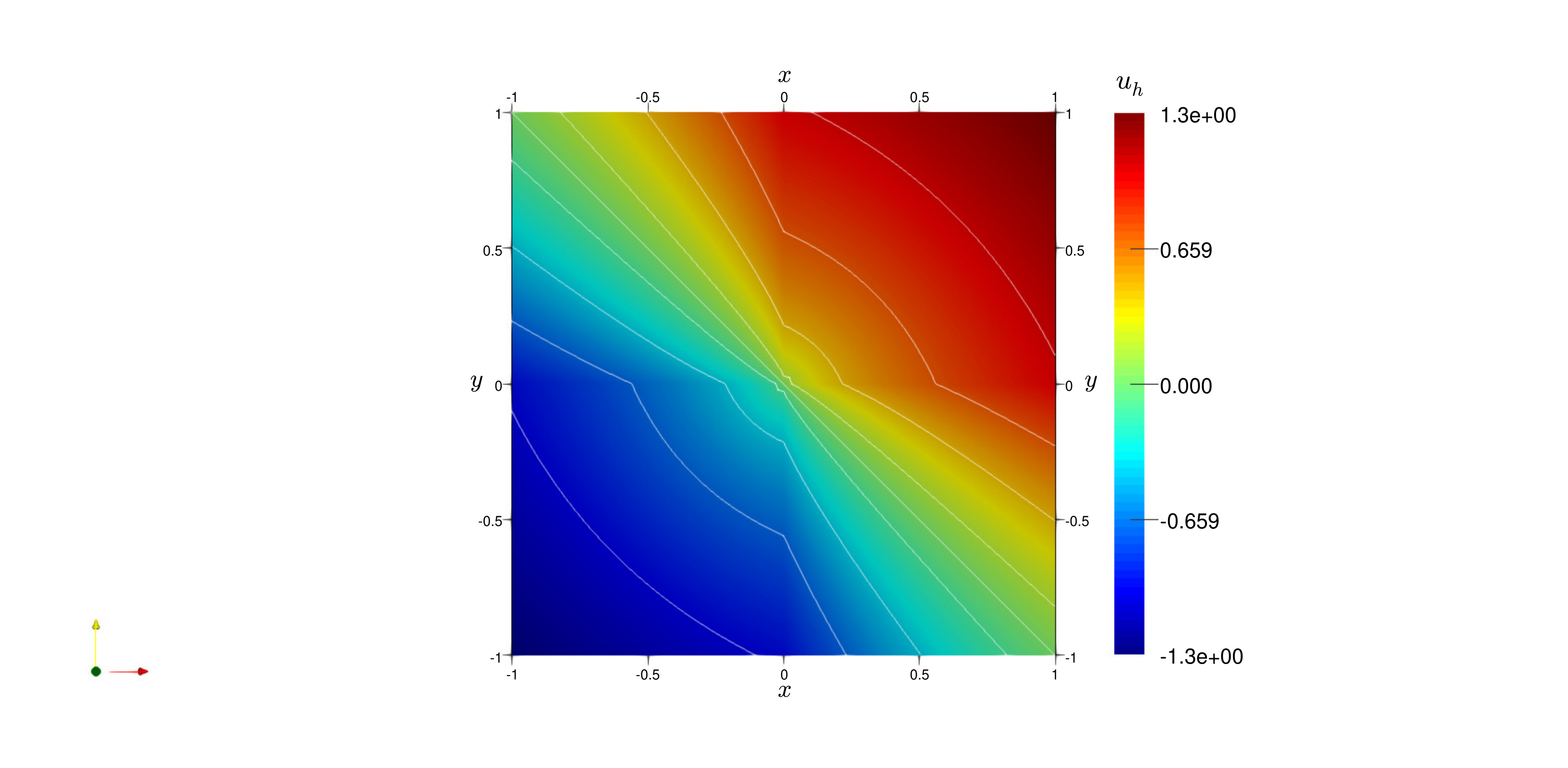}
  \end{subfigure}
  \caption{Test 2 -- (Left) A non-uniform triangular mesh. (Right) The SUIP-DG solution for the Kellogg benchmark using piecewise linear polynomials ($k=1$).}
  \label{fig:both}
\end{figure}

\section{Conclusions}
\label{sec:conclusions}

We have derived a primal discontinuous Galerkin formulation for heterogeneous and anisotropic diffusion by exact elimination of the skeletal unknown in a compact hybridized interior penalty (H-IP) method. The resulting UIP-DG formulation is entirely determined by the hybrid stabilization parameter $\tau$: the transmissibility-based weights $\omega_i^\tau$, the primal jump penalization $\varrho_0$, and the flux-jump stabilization $\varrho_1$ are obtained in closed form. In particular, all interface terms arise from the hybrid structure rather than being introduced \emph{ad hoc}, and the associated numerical flux is single-valued across interfaces, ensuring local conservation for all variants $\epsilon\in\{0,\pm1\}$.

The stability analysis establishes coercivity and boundedness in a discrete energy norm under the condition $\alpha_0 > 2\eta_0$, with constants that are uniform with respect to both the mesh size $h$ and the diffusion contrast $\kappa_{\max}/\kappa_{\min}$. A notable feature is that the coercivity threshold is independent of the discrete trace inequality constant $C_T$, reflecting the exact cancellation properties induced by the transmissibility-based scaling. Optimal energy-norm convergence rates are obtained for all variants, and an optimal $\LL{2}$-error estimate is proved for the symmetric case via an Aubin--Nitsche argument. These theoretical results are fully confirmed by numerical experiments on heterogeneous anisotropic benchmarks.

The present framework suggests two natural directions for future work. First, restricting the primal space to a conforming subspace eliminates primal jump contributions and yields a formulation based solely on flux-jump stabilization, thereby recovering classical continuous interior penalty (CIP) methods within the same hybridization framework. This establishes a direct connection between UIP-DG and CIP-type stabilizations, with penalty parameters inherited from transmissibility scaling rather than prescribed heuristically. Second, the extension to (locally degenerate) advection--diffusion--reaction problems constitutes a natural continuation of this work. Starting from the H-IP framework with Péclet-dependent stabilization, the UIP-DG reduction is expected to yield advection-corrected transmissibility weights and intrinsically asymmetric formulations reflecting the non-self-adjoint nature of the total flux. In degenerate regimes, the penalty structure adapts automatically without additional modeling assumptions, highlighting the robustness of the transmissibility-based construction. These extensions are the subject of ongoing work.

\section*{Acknowledgments}

The second author gratefully acknowledges Professor Béatrice Rivière for drawing his attention to the reference \cite{FabienNM2019} during her visit to the University of La Réunion in June 2023. He is also indebted to Professor Erik Burman for his guidance and valuable advice in establishing the proof of item (ii) in \cref{lem:extended-interface-bounds}.

\bibliographystyle{elsarticle-num}
\bibliography{references_uip}

\end{document}